\documentclass{article}

 \usepackage{graphicx}
\usepackage{hyperref}
\usepackage{indentfirst,csquotes}
\usepackage{etoolbox}
\usepackage{subcaption}
\usepackage{mathrsfs}
\usepackage{blkarray}
\usepackage{amsmath,amssymb,amsthm,enumerate,mathrsfs}
\newtheorem{theorem}{Theorem}[section]
\newtheorem{corollary}[theorem]{Corollary}
\newtheorem{lemma}[theorem]{Lemma}
\newtheorem{proposition}[theorem]{Proposition}
\theoremstyle{definition}
\newtheorem{definition}[theorem]{Definition}

\newtheorem{remark}[theorem]{Remark}

\usepackage{algorithm, algorithmic}

\makeatletter
\renewcommand{\p@algorithm}{\arabic{algorithm}\expandafter\@gobble}
\makeatother 

\newcommand{\INPUT}{\item[\textbf{Input:}]}
\newcommand{\OUTPUT}{\item[\textbf{Output:}]}
\newcounter{step}[algorithm]
\newcommand\STEP[2][\(\triangleright\)]{%
	\refstepcounter{step}
	\vskip 0.25\baselineskip
	\item[]\hskip -\algorithmicindent #1 \textbf{Step \arabic{step}}%
	\ifthenelse{\equal{\unexpanded{#2}}{}}{}{ (\texttt{#2})}%
	\textbf{.}%
}
\newenvironment{algor}{\algor}{}
\def\algor#1\end{%
	\noindent\fbox{%
	\begin{minipage}[b]{\dimexpr\columnwidth-\algorithmicindent\relax}
	\begin{algorithmic}
	#1
	\end{algorithmic}
	\end{minipage}
	}%
\end}

\numberwithin{equation}{section}

\DeclareMathOperator{\rank}{rank}

\begin{document}
\makeatletter

\begin{center}
\large{\bf Linear programming for finite-horizon vector-valued Markov decision processes}
\end{center}\vspace{5mm}
\begin{center}

\textsc{Anas Mifrani, Dominikus Noll}\end{center}

\vspace{2mm}

\footnotesize{
\noindent\begin{minipage}{14cm}
{\bf Abstract:}
We propose a vector linear programming formulation for a non-stationary, finite-horizon Markov decision process
with vector-valued rewards. Pareto efficient policies are shown to correspond to
efficient solutions of the linear program, and vector linear programming theory allows us
to fully characterize deterministic efficient policies. 
An algorithm for enumerating all efficient deterministic policies is presented then tested numerically in an engineering application.
\end{minipage}
 \\[5mm]

\noindent{\bf Keywords:} {Multi-objective Markov decision process, efficient policy, multi-objective linear program, efficient vertex}\\
\noindent{\bf Mathematics Subject Classification:} {90C40, 90C05, 90C29}

\hbox to14cm{\hrulefill}\par






\section{Introduction}
\label{sect_intro}
It is not uncommon in practical Markov decision process (MDP) models that 
the decision maker has to cope with several -- sometimes competing -- objectives which cannot readily be subsumed under a single, scalar-valued reward function.  
%
%
This has been a strong incentive for
researchers to investigate \textit{vector-valued Markov decision processes} (vMDPs); see, e.g., \cite{8, 21, 20, 10, 17}.
A central task has been to expand MDP theory to the vector-valued case. This is by no means straightforward, as the first author has recently shown. He finds, for example, that not all solutions of the functional equations for finite vMDPs \cite{21} are achieved by Markov policies, contrary to what obtains in the corresponding scalar-valued case \cite{13}.  
%
%

It is well known that for criteria which are linear in the rewards, an MDP can be formulated and optimized as a linear program (LP) (see \cite{16} for an example in the infinite-horizon case, and \cite{4} for its finite-horizon analogue). This raises the question whether, similarly, \textit{vector linear programming} (vLP) may be employed in the optimal control of vMDPs. For infinite-horizon vMDPs, Novák \cite{15} has shown that this is indeed possible. However, such a result has yet to be presented for the finite horizon case, and this is not completely surprising in view of the above-mentioned discrepancy between finite MDPs and vMDPs. Presently we shall establish a vLP approach to the optimal control of finite-horizon vMDPs. 

Puterman \cite{16} highlights two features of linear programming that make it preferable to other traditional methods like value or policy iteration. First,
state and control constraints, reflecting limitations on resources and budget, quality standards, or long term viability quests, can be built naturally into the model.  Second, 
the well-explored machinery of LPs offers readily available algorithmic and analytic tools for assessing MDPs. For instance, sensitivity analysis theory \cite{5}, when applied to MDPs, can help quantify how sensitive optimal policies are with regard to changes in the reward function or to modifications in the constraints.
This makes it highly desirable to exploit in much the same way properties of vLPs when solving vMDPs.

Vector linear programs have been studied in a variety of contributions. For instance, Benson 
\cite{2}, Steuer \cite{18} and Yu \cite{23}
investigate existence and structure of efficient solutions; Armand and Malivert
\cite{1} and Evans and Steuer \cite{7} each
develop methods for enumerating efficient solutions or generating a representative sample of the efficient set; and Geoffrion \cite{9}, Evans and Steuer \cite{7}, Benson \cite{2} as well as Ecker and Kouada \cite{6} examine the relation of vLPs to scalar LPs. 
Testing vertices of the constraint polyhedron $P$ for efficiency,
or proceeding from one efficient vertex to another adjacent efficient 
one,
have for instance been addressed
by each of Evans and Steuer \cite{7} and Ecker and Kouada \cite{6}.

The structure of the paper is as follows. 
In Section \ref{sect_problem}, the vMDP is presented in detail. The vLP formulation is obtained in Section \ref{sect_LP}, and
the main result (Theorem \ref{centralthm}) gives equivalence between the controlled vMDP 
and the vLP. Section \ref{sect_consequences} draws 
consequences from the main theorem. In Sections \ref{sect_regular_case} and \ref{sect_general_case}
efficient deterministic policies are shown to correspond to 
efficient basic feasible solutions of the vLP.
Section \ref{sect_algo} presents our algorithm for detecting all efficient deterministic policies. Numerical results demonstrating the usefulness of the algorithm are reported in Section \ref{sect_numeric}. 
\\

\section{Problem setting}
\label{sect_problem}
Let $\geqq$ be the standard product order on $\mathbb{R}^k$, 
and let $V \subseteq \mathbb{R}^k$. We say an element $v \in V$ is \textit{efficient} in $V$, and write $v \in \mathscr E(V)$, if it is maximal with regard to $\geqq$, i.e., if
there exists no $v' \in V$ with $v'\not= v$ such that $v' \geqq v$. If $v, v'\in \mathscr E(V)$ with $v\not= v'$, 
then $v$ and $v'$ are incomparable, i.e., $v \ngeqq v'$ and $v' \ngeqq v$. 

Given a set $P \subseteq \mathbb{R}^n$ and a vector-valued mapping $f: P \to \mathbb{R}^k$, we call 
the problem of simultaneously maximizing $f_1(x),\dots,f_k(x)$
over all $x\in P$, written as 
\begin{equation*}
\begin{aligned}
\operatorname{V-max} \quad & f(x) = (f_1(x), ..., f_k(x)), & \textrm{s.t.} \quad x \in P,\\
\end{aligned}
\end{equation*}
a \textit{vector maximization program}. The goal is to find feasible points $x^*\in P$ such that $f(x^*) \in \mathscr E(f(P))$.  
We call such $x^*$ \textit{efficient solutions},
the set of all efficient solutions being  $P_E$. We call $f(P)$ the {\it value set} of the program,
and $\mathscr{E}(f(P))=f(P_E)$ its \textit{efficient value set}. In the special case where $P$ is a convex polyhedron and the criteria $f_i(x) = c_i^{T}x$ 
are linear, we obtain a multi-objective or vector linear program (vLP)
\begin{equation}
\tag{vLP}
\label{vLP}
\begin{aligned}
\operatorname{V-max} \quad & Cx = (c_1^{T}x, ..., c_k^{T}x)^T, & \textrm{s.t.} \quad x \in P,\\
\end{aligned}
\end{equation}
where $C$ is a $k \times n$-matrix gathering the criteria as $Cx$.  
Here the polyhedron typically has the canonical form
$P = \{x \in \mathbb{R}^n: Ax = b,\, x \geqq 0\}$, with $A$ an $m \times n$ matrix and $b \in \mathbb{R}^m$ a given vector such that $m < n$.
%

The decision model we study in this paper is the following. A controller observes a system at $T$ discrete epochs, $T \geqq2$. At each decision epoch $t = 1, ..., T-1$ the system is in one of a finite number of states, and a control action is selected accordingly. Let $\mathbb{S} = \{1, ..., S\}$ be the set of all states, and for $s \in \mathbb{S}$, let $\mathbb{A}_s = \{1, ..., k_s\}$ be the finite set of actions available in state $s$, with $k_s > 1$ for at least one $s$. Taking an action $a$ in a state $s$ at a decision epoch $t$ generates a vector-valued reward $R_t(s, a) \in \mathbb{R}^k$, and results in the process occupying a new state $j$ at epoch $t+1$ with probability $p_t(j | s, a)$. No actions are taken at epoch $T$, but there is a reward for terminating in state $s\in \mathbb{S}$, denoted by $R_T(s)$. We choose $\alpha(s)$, $s \in \mathbb{S}$, to be positive scalars satisfying $\sum_{s \in \mathbb{S}} \alpha(s) = 1$, with the interpretation that $\alpha(s)$ is the probability that the process starts in state $s$. We refer to this model as a vMDP. 

When the controller takes an action at a decision epoch $t$, they are implementing some \textit{decision rule} $\pi_t$. We construe $\pi_t$ as a mapping from $\mathbb{S}$ to the set of all probability distributions on $\mathbb A = \bigcup_{s \in \mathbb{S}} \mathbb{A}_s$, such that in state $s$ actions $a \in \mathbb{A}_s$ are selected with probability $q(a \,|\, s,\, \pi_t)$. 
In the special case where, for each state $s$, an action $a$ exists which satisfies 
$q(a \,|\, s,\, \pi_t) = 1$, we say that $\pi_t$ is {\it deterministic}.

We call the vector $\pi = (\pi_1, ..., \pi_{T-1})$ 
gathering these decision rules a 
\textit{policy}. Deterministic policies, i.e., policies where all the $\pi_t$ are deterministic, are of particular interest, both theoretically and computationally \cite{16}. We let $\Pi$ denote the set of all 
policies, $\Pi^{D}$ the subset of deterministic policies, $\Pi^E$ the set of all efficient policies, and $\Pi^{ED}$ the set of all efficient
deterministic policies.

Each $\pi = (\pi_1, ..., \pi_{T-1})$ induces a probability measure $\mathbb{P}^{\pi}$ over the space of state-action trajectories through 
\begin{equation*}
\mathbb{P}^{\pi}(\textbf{S}_1 = s) = \alpha(s),
\end{equation*}
\begin{equation*}
\mathbb{P}^{\pi}(\textbf{A}_t = a \,|\, \textbf{S}_1 = s_1,\, \textbf{A}_1 = a_1,\, \textbf{S}_2 = s_2, ...,\, \textbf{A}_{t-1} = a_{t-1},\, \textbf{S}_t = s) = q(a\,|\,s,\, \pi_t),
\end{equation*}
\begin{equation*}
\mathbb{P}^{\pi}(\textbf{S}_{t+1} = j \,|\, \textbf{S}_1 = s_1,\, \textbf{A}_1 = a_1,\, \textbf{S}_2 = s_2, ...,\, \textbf{A}_{t-1} = a_{t-1},\, \textbf{S}_t = s,\, \textbf{A}_t = a) = p_t(j | s, a),
\end{equation*}
where $\textbf{S}_t$ and $\textbf{A}_t$ denote, respectively, the random state and action at epoch $t$. Define the expected total reward accrued over the lifetime of the process if the policy implemented is $\pi$ and the process initially occupies state $s$ as 
\begin{equation*}
v^{\pi}(s) = \mathbb{E}_{s}^{\pi}\biggl[\sum_{t = 1}^{T-1}R_t(\textbf{S}_t, \textbf{A}_t) + R_T(\textbf{S}_T)\biggr],
\end{equation*} 
where the expectation is taken with respect to $\mathbb{P}^{\pi}$, and where it is implicit that $\textbf{A}_t$ follows $\pi_t(\textbf{S}_t)$ for all $t < T$.


With this as background, $\sum_{s \in \mathbb{S}}\alpha(s)v^{\pi}(s)$ represents the expected total reward accrued under $\pi$ if the initial state is distributed according to $\alpha$. We desire policies which yield as large an expected total reward as possible; that is, we wish to solve the vector maximization program 
\begin{equation}
\tag{vMDP}
\begin{aligned}
\label{program}
\operatorname{V-max} \quad & \sum_{s \in \mathbb{S}}\alpha(s)v^{\pi}(s), & \textrm{s.t.} \quad \pi \in \Pi,\\
\end{aligned}
\end{equation}
where we keep for the ease of notation the acronym vMDP for the optimal control problem associated with the vMDP.
We let $V = \bigl\{\sum_{s \in \mathbb{S}}\alpha(s)v^{\pi}(s):\ \pi \in \Pi\bigr\}$ denote the value set of (\ref{program}). 
Efficient solutions $\pi^*$ of (\ref{program}) will be called efficient policies, the set of efficient policies being $\Pi^E$.

The goal of this work is to transform (\ref{program}) 
into an equivalent vLP. 
Due to its more amenable structure, 
the latter may then be used favorably to analyze the vMDP. In particular, 
we will show that under this transformation, policies
are mapped to feasible solutions of the vLP, efficient policies correspond to efficient solutions of the vLP, and most importantly, efficient deterministic policies are in one-to-one correspondence with the basic feasible solutions, or vertices,  of the vLP.

Transforming (\ref{program}) into  an equivalent vLP has, among others, the
benefit that
the onerous task of 
enumerating efficient deterministic policies can be carried out 
via the vLP, where this reduces to
enumerating efficient vertices of $P$. 
Since the efficient solution set $P_E$ of the vLP is a union of faces and generally a {\it non-convex but connected subset of the relative boundary} of $P$ \cite{3}, 
finding all efficient vertices by an exhaustive search is possible
for moderate problem sizes, and we address this algorithmically.


Enumeration of efficient vertices requires
computational tests for determining whether a candidate
vertex given by its basic feasible solution is efficient. Evans and Steuer \cite{7} present a necessary and sufficient condition for efficiency 
in terms of the boundedness of an auxiliary LP. In
the same vein, Ecker and Kouada \cite{6} give a test of
efficiency of an edge incident to a vertex already known as efficient. 
These authors also devise a procedure based on their characterization for detecting all efficient vertices. While Evans and Steuer do not explain how their condition can be leveraged in the context of a search for efficient basic feasible solutions, it can easily be incorporated into Ecker and Kouada's algorithm as a substitute for the efficient edge characterization.
Therefore, as we will demonstrate,
a combination of these methods 
can be used to generate the set $\Pi^{E} \cap \Pi^{D}= \Pi^{ED}$. 

The latter is of significance, 
because as Mifrani \cite{13} has recently shown, the vector extension of the functional equations of finite-horizon MDPs \cite{21} fails to give
an algorithm for constructing efficient Markovian deterministic policies. This leaves an unexpected  gap, which we 
close in this work. 
\\

\section{Linear programming formulation}
\label{sect_LP}
Let $\pi$ be a policy. 
If we let $u_t^{\pi}(s)$ denote the expected total reward generated by $\pi$ from time $t$ onward if the state at that time is $s$, i.e., \begin{equation*}u_t^{\pi}(s) = \mathbb{E}_{s}^{\pi}\biggl[\sum_{i = t}^{T-1}R_i(\textbf{S}_i, \textbf{A}_i) + R_T(\textbf{S}_T)\biggr],\end{equation*} with $u_T^{\pi}(s) = R_T(s)$, then simple probability operations yield the recursion \begin{equation*}u_t^{\pi}(s) = \sum_{a \in \mathbb{A}_s}q(a \,|\, s,\ \pi_t)\biggl[R_t(s, a) + \sum_{j \in \mathbb{S}}p_t(j | s, a)u_{t+1}^{\pi}(j)\biggr]\end{equation*} for all $t = 1, ..., T-1$ and $s \in \mathbb{S}$. By expanding this expression over all future states and epochs, and by noting that $u_1^{\pi}(s) = v^{\pi}(s)$, we obtain
\begin{align}
\label{value}
\begin{split}
v^{\pi}(s) = &\sum_{t = 1}^{T-1}\sum_{j \in \mathbb{S}}\sum_{a \in \mathbb{A}_j} \mathbb{P}^{\pi}(\textbf{S}_t = j, \textbf{A}_t = a \,|\, \textbf{S}_1 = s)R_t(j, a) 
\hspace*{1cm}\\& \hspace*{5cm}+ \sum_{j \in \mathbb{S}}\mathbb{P}^{\pi}(\textbf{S}_T = j\,|\, \textbf{S}_1 = s)R_T(j).
\end{split}
\end{align}
Multiplying by $\alpha(s)$ then summing over the states yields
\begin{equation}\label{eq1}\sum_{s \in \mathbb{S}}\alpha(s)v^{\pi}(s) = \sum_{t = 1}^{T-1}\sum_{s \in \mathbb{S}}\sum_{a \in \mathbb{A}_s} x_{\pi, t}(s, a)R_t(s, a) + \sum_{s \in \mathbb{S}}x_{\pi, T}(s)R_T(s),\end{equation}
where we put
\begin{equation}\label{eq2}x_{\pi, t}(s, a) = \sum_{j \in \mathbb{S}}\alpha(j)\,\mathbb{P}^{\pi}(\textbf{S}_t = s, \textbf{A}_t = a \,|\, \textbf{S}_1 = j)\end{equation} for all $s \in \mathbb{S}$ and $a \in \mathbb{A}_s$ if $t = 1, ..., T-1$, and \begin{equation}\label{eq3}x_{\pi, T}(s) = \sum_{j \in \mathbb{S}}\alpha(j)\,\mathbb{P}^{\pi}(\textbf{S}_T = s \,|\, \textbf{S}_1 = j)\end{equation} for all $s \in \mathbb{S}$.
We shall often call $x_\pi$ a vector, where in fact components have triple indices
$t<T, s\in \mathbb S, a \in \mathbb{A}_s$.

Equation (\ref{eq1}) and the mapping $\pi \mapsto x_\pi$
defined through
(\ref{eq2}) and (\ref{eq3}) shall play a central role in the subsequent development. If $t < T$, $x_{\pi, t}(s, a)$ 
is the joint probability under policy $\pi$ and initial state distribution $\alpha$ that state $s$ is occupied and action $a$ is selected at epoch $t$. If $t = T$, then $x_{\pi, T}(s)$ is the probability that the process terminates in state $s$ under $\pi$ and $\alpha$.


\begin{proposition}
\label{prop1}
For any policy $\pi\in \Pi$, the vector $x_\pi$ satisfies $x_{\pi} \geqq 0$ and the following relations:
\begin{itemize}
\item[\rm{(a)}] For all $j \in \mathbb{S}$, $\sum_{a \in \mathbb{A}_j} x_{\pi, 1}(j, a) = \alpha(j)$.
\item[\rm{(b)}] For all $t = 1, ..., T-2$ and $j \in \mathbb{S}$, 
\begin{equation*}
\sum_{a \in \mathbb{A}_j} x_{\pi, t+1}(j, a) = \sum_{s \in \mathbb{S}}\sum_{a \in \mathbb{A}_s} p_t(j | s, a)x_{\pi, t}(s, a);
\end{equation*}
\item[\rm{(c)}] For all $j \in \mathbb{S}$, $x_{\pi, T}(j) = \sum_{s \in \mathbb{S}}\sum_{a \in \mathbb{A}_s} p_{T-1}(j | s, a)x_{\pi, T-1}(s, a)$;
\end{itemize}
\end{proposition}

\begin{proof}
%
That $x_{\pi} \geqq 0$ is obvious from (\ref{eq2}) and (\ref{eq3}). With regard to part (a), we have
\begin{flalign*}
\sum_{a \in \mathbb{A}_j} x_{\pi, 1}(j, a) &= \sum_{a \in \mathbb{A}_j} \sum_{i \in \mathbb{S}}\alpha(i)\mathbb{P}^{\pi}(\textbf{S}_1 = j, \textbf{A}_1 = a \,|\, \textbf{S}_1 = i)\\
&= \sum_{i \in \mathbb{S}}\alpha(i)\mathbb{P}^{\pi}(\textbf{S}_1 = j \,|\, \textbf{S}_1 = i)\\
&= \alpha(j).
\end{flalign*}

Next, for part (b) we have

\begin{flalign*}
\sum_{s \in \mathbb{S}}\sum_{a \in \mathbb{A}_s} p_t(j | s, a)x_{\pi, t}(s, a) &= \sum_{s \in \mathbb{S}}\sum_{a \in \mathbb{A}_s} p_t(j | s, a)\sum_{i \in \mathbb{S}}\alpha(i)\,\mathbb{P}^{\pi}(\textbf{S}_t = s, \textbf{A}_t = a \,|\, \textbf{S}_1 = i)\\
&= \sum_{i \in \mathbb{S}}\alpha(i)\sum_{s \in \mathbb{S}}\sum_{a \in \mathbb{A}_s} \mathbb{P}^{\pi}(\textbf{S}_{t+1} = j, \textbf{S}_t = s, \textbf{A}_t = a \,|\, \textbf{S}_1 = i)\\
&= \sum_{a \in \mathbb{A}_j}\sum_{i \in \mathbb{S}}\alpha(i)\mathbb{P}^{\pi}(\textbf{S}_{t+1} = j, \textbf{A}_{t+1} = a\,|\, \textbf{S}_1 = i)\\
&= \sum_{a \in \mathbb{A}_j}x_{\pi, t+1}(j, a).
\end{flalign*}

Part (c) follows from a calculation similar to that employed in part (b).
\end{proof}

As we shall see, 
Equation (\ref{eq1}) in tandem with Proposition \ref{prop1} provides the basis for the vLP we seek. 
Proposition \ref{prop1}
suggests the following

\begin{definition}
{\rm (State-action frequencies)}. 
Vectors $x \geqq 0$ satisfying conditions {\rm (a)-(c)} in Proposition
{\rm \ref{prop1}} are called {\rm state-action frequency vectors}.  The quantities
$x_t(j,a)$ are the state-action frequencies. The set
of all state-action frequency vectors is denoted by $P$.
\end{definition}

In this notation Proposition \ref{prop1} tells us that 
the correspondence $\pi \mapsto x_\pi$  defined via (\ref{eq2}) and (\ref{eq3}) 
maps policies $\pi \in \Pi$ into state-action
frequency vectors $x_\pi \in P$. 
The question is therefore whether state-action frequencies already
determine the policy, i.e., whether all state-action frequency vectors $x$ can be realized by a policy
as $x=x_\pi$. In order to investigate this, we have
to consider the notion of {\it reachability} of a state by a policy.

It may happen that a state $s$ is not reachable at an epoch $t$ 
under a policy $\pi$, i.e., 
$\mathbb P^\pi(\textbf{S}_t=s|\textbf{S}_1=j)=0$ for all initial states $j\in \mathbb{S}$. In that case the decisions $a$ 
which $\pi$ makes at epoch $t$ with probabilities $q(a|s,\pi_t)$ 
have
no real influence on the relevant system trajectories. More precisely, when we define
a policy $\pi'$ which differs from $\pi$ only at epoch $t$ and state $s$, that is,
$q(\cdot|s,\pi_{t}')\not= q(\cdot|s,\pi_t)$ at $(s,t)$, while $q(\cdot|s',\pi_{t'}') = q(\cdot|s',\pi_{t'})$ elsewhere, we only modify state-action trajectories occurring with 
zero probability. 
Therefore, the fact that
rewards $R_t(s,a)$ are accrued differently at $(s,t)$ according to $\pi$ or $\pi'$
will  not matter at all, because these
are awarded with zero probability either way. In short, we have
$v^\pi(s) = v^{\pi'}(s)$ for all $s$.

We use this observation to fix decisions irrelevant in this sense in a unique way. 
\begin{definition}
    A policy $\pi$ is {\rm regular} if it satisfies the following
    condition: For all decision epochs $t = 1, ..., T-1$ and states $s$, if $\mathbb P^\pi(\textbf{S}_t=s|\textbf{S}_1=j)=0$ for all $j\in \mathbb{S}$, 
    then $q(1|s,\pi_t)=1$. The set of all regular policies is denoted by $\Pi_r$.
\end{definition}

A given policy $\pi$ may be modified at unreachable
state-epoch pairs $(s,t)$ by setting $q(1|s,\pi_t) = 1$, which 
leads to an equivalent regular policy $\pi'$. This gives an equivalence relation
$\sim$ on the set of all policies, where $\pi_1 \sim \pi_2$ iff both have the same regularization.
In this sense $\Pi_r$ represents the quotient set $\Pi/_\sim$.
From the construction we can see that $\pi_1 \sim \pi_2$ implies 
$v^{\pi_1}(s) = v^{\pi_2}(s)$ for all $s$.
This has the important consequence that we may restrict optimization
in program (\ref{program}) to policies in $\Pi_r$. This will do no harm, as
every efficient $\pi\in \Pi^E$ is equivalent to a regular efficient $\pi_r \in \Pi_r \cap \Pi^E
= \Pi_r^E$.

Unreachable states are caused by the internal dynamics of the process. For example, if $s \in \mathbb{S}$ has $p_{t-1}(s|s',a')=0$ for all $s'\in \mathbb{S}$ and all $a'\in \mathbb{A}_{s'}$, then the probability that any policy will reach $s$ at time $t$ is zero, irrespective of the actions prescribed by that policy at the preceding epochs. However, even if the process allows transition to state $s$, we may find, as the next lemma shows, policies unable to reach it.
\begin{lemma}
\label{unreach1}
    The following are equivalent:
    \begin{enumerate}
       \item[{\rm (i)}] 
       Some policies have unreachable state-epoch pairs. That is, there exist at least one policy $\pi$ and at least one state-epoch pair $(s, t)$ such that for all initial states $j \in \mathbb{S}$, $\mathbb P^\pi(\textbf{S}_t=s|\textbf{S}_1=j)=0$.
       \item[{\rm (ii)}] There exist a state $s$ and an epoch $t = 2, \dots, T$ such that for every state $s' \in \mathbb{S}$ there is at least one $a_{s'}\in \mathbb{A}_{s'}$  with $p_{t-1}(s|s',a_{s'})=0$.
    \end{enumerate}
\end{lemma}
\begin{proof}
If $t = 2, ..., T$ is an epoch and $s, j$ are states, we denote by $\mathscr T$ the set of all state-action trajectories starting at $j$ in the first epoch and arriving at $s$ at epoch $t-1$, 
that is, the set of all $\tau = (j, a_1, s_2, \dots, a_{t-1}, s)$ where $a_1 \in \mathbb{A}_{j}$, $s_2 \in \mathbb{S}$, \dots, $a_{t-1} \in \mathbb{A}_{s_{t-1}}$. With this notation, the chain rule of conditional probability 
yields
\begin{multline*}
\label{prob_path}
\mathbb{P}^{\pi}(\textbf{S}_t = s | \textbf{S}_1 = j) = \sum_{\tau \in \mathscr T} q(a_1 | j, \pi_1)p_1(s_2 | j, a_1)q(a_2 | s_2, \pi_2)p_2(s_3 | s_2, a_2)\\\dots q(a_{t-1} | s_{t-1}, \pi_{t-1})p_{t-1}(s | s_{t-1}, a_{t-1}).
\end{multline*}
Now if (ii) is satisfied, then it suffices to define $\pi$ such that
$q(a_{s'}|s',\pi_{t-1}) = 1$ for all $s' \in \mathbb{S}$, so that in each trajectory $\tau$ we have a zero factor.

Conversely, assuming (i), there exists a first time  $t$ such that 
some policy cannot reach some $(s,t)$. That means all policies can reach all states
at all times 
$t' \leq t-1$.
The policy $\pi$ which does not reach $(s,t)$ can hence reach all $(j, t-1)$. Suppose that for
one such $j$ we have $p_{t-1}(s|j,a) > 0$ for all $a\in \mathbb{A}_j$. Then as $\pi$ has to pick some $a\in \mathbb{A}_j$ with probability $>0$, it can reach $s$ at time $t$, as follows again from the chain rule above, 
a contradiction. Hence there must exist for every $j$ some $a_j\in \mathbb{A}_j$
such that $p_{t-1}(s|j,a_j)=0$. 
\end{proof}

\begin{lemma}
\label{unreach2}
    The following are equivalent:
    \begin{enumerate}
        \item[{\rm (i')}] Every policy $\pi$ has unreachable state-epoch pairs.
        \item[{\rm (ii')}] There exist a state $s$ and an epoch $t$ such that
        $p_{t-1}(s|s',a')=0$ for every $s'$ and every $a'\in \mathbb{A}_{s'}$. Equivalently, $(s,t)$
cannot be reached by any policy.    
\end{enumerate}
\end{lemma}

\begin{proof}
We consider the digraph $\vec{G}=(N,A)$ of nodes $(s,t)$, where we place an arc from
    $(s',t-1)$ to $(s,t)$ iff $\sum_{a'\in \mathbb{A}_{s'}}p_{t-1}(s|s',a') > 0$. 
    This digraph only depends on the process dynamics.

    Now we interpret a policy $\pi$ as selecting among the arcs of $\vec{G}$ only
    those where it gives one of the $a'$
    a positive probability.
  In other words,  $\pi$ selects a subgraph $\vec{G}_\pi=(N,A_\pi)$ of the digraph $\vec{G}$ by letting an
arc of $\vec{G}$ remain in the subgraph $\vec{G}_\pi$
iff $\sum_{a'\in \mathbb{A}_{s'}} p_{t-1}(s|s',a')q(a'|s',\pi_{t-1}) > 0$.
    Then $\pi$ reaches $(s,t)$ iff there exists a directed path from any one of the $(s',1)$ to $(s,t)$
    in $\vec{G}_\pi$, as follows directly from the chain rule of conditional probability.
    
With this interpretation, condition 
    (ii') says that there is no arc from $(s',t-1)$ to $(s,t)$ in $\vec{G}$, and then obviously there is no path
    reaching $(s,t)$ in $\vec{G}$ from the first layer $(j,1)$, as such a path has to go through layer $t-1$. Then obviously neither does such a path exist in any $\vec{G}_\pi$. 
    Hence (ii') implies (i').

    On the other hand, suppose every $\pi$ has unreachable state-epoch pairs, and yet
    (ii') is false.
    The negation of (ii') says
    that for every $(s,t)$ there exists $s'\in \mathbb{S}$
    and $a' \in \mathbb{A}_{s'}$ such that $p_{t-1}(s|s',a') > 0$. In other words,
    there is an arc from $(s',t-1)$ to $(s,t)$ in $\vec{G}$. But we can apply the argument to
    $(s',t-1)$ again, i.e.,  there exist $s''$, $a''\in \mathbb{A}_{s''}$ such that $p_{t-2}(s'|s'',a'')>0$, 
    hence
    there is an arc from $(s'',t-2)$ to $(s',t-1)$.  Going backwards, 
    this gives for every $(s,t)$ a path in $\vec{G}$ from some $(j,1)$ to $(s,t)$.
    In consequence, if we take
    a policy
    $\pi$ which is mixing in the sense that $q(a|s,\pi_{t}) > 0$ for every
    $(s,t)$ and every $a\in \mathbb{A}_s$, 
then $\vec{G}=\vec{G}_\pi$,
    so that every $(s,t)$ can be reached in $\vec{G}_\pi$, a contradiction.
    \end{proof}





When every $\pi$ has unreachable state-epoch pairs, 
there is one $(s,t)$ which none of the $\pi$ can reach. Consequently, if some but not all $\pi$ have unreachable $(s,t)$, then at every
time $t$ there must exist some $(s,t)$ which all $\pi$ can reach.

We say that the vMDP is {\it regular} when situation (i) in Lemma \ref{unreach1} is excluded, that is, when all policies can reach all states at every stage of the decision making. Condition (ii) 
furnishes a simple test of regularity. 
When the process is regular, all policies are regular, i.e.,
$\Pi = \Pi_r$. 
Conversely,
if $k_s > 1$ for all $s\in \mathbb{S}$, then $\Pi = \Pi_r$ implies regularity of the process.

We now proceed to show that
$\pi \to x_\pi$ is in fact a bijection from regular policies
$\Pi_r$ onto state-action frequency vectors $P$. 
In other words, all state-action frequency vectors $x\in P$
arise as $x=x_\pi$ for certain $\pi \in \Pi_r$, and this representation  is unique.
An analogous result for the scalar-valued infinite-horizon case is presented in \cite[pp. 225-226]{16},
where considerations of regularity are unnecessary.  Our proof is constructive and is 
divided into several lemmas and propositions.

As a
first step we show  that $\pi \mapsto x_\pi$ is injective, i.e., 
distinct regular policies $\pi$ give rise to distinct
state-action frequency vectors $x_\pi$. This requires the following
two lemmas.
\begin{lemma}
\label{lem1}
Let $\pi \in \Pi$. The following statements are true:
\begin{itemize}
    \item[{\rm (i)}] The state $s \in \mathbb{S}$ is reachable at epoch
    $t=1,\dots,T-1$ under $\pi$ iff  $\sum_{a \in \mathbb{A}_s} x_{\pi, t}(s, a) > 0$. 
    \item[{\rm (ii)}] The state $s \in \mathbb{S}$ is reachable at epoch $t=T$
    under $\pi$ iff $x_{\pi, T}(s) > 0$.
\end{itemize}
\end{lemma}

\begin{proof}
(i) When $s$ is reachable at epoch $t$, 
there exists $j \in \mathbb{S}$ satisfying $\mathbb{P}^{\pi}(\textbf{S}_t = s | \textbf{S}_1 = j) > 0$, and therefore $\alpha(j)\mathbb{P}^{\pi}(\textbf{S}_t = s | \textbf{S}_1 = j) > 0$. Thus, 
\begin{flalign*}
\sum_{a \in \mathbb{A}_s} x_{\pi, t}(s, a) &= \sum_{a \in \mathbb{A}_s} \sum_{i \in \mathbb{S}}\alpha(i)\mathbb{P}^{\pi}(\textbf{S}_t = s, \textbf{A}_t = a \,|\, \textbf{S}_1 = i)\\
&= \sum_{i \in \mathbb{S}}\alpha(i)\mathbb{P}^{\pi}(\textbf{S}_t = s \,|\, \textbf{S}_1 = i) > 0.
\end{flalign*}
On the other hand, if $s$ is not reachable, then all $\mathbb P^\pi(\textbf{S}_t=s|\textbf{S}_1=i)$, $i \in \mathbb{S}$,
vanish, and then the left hand side vanishes, too.

(ii) Again, when $s$ is reachable at $t=T$, there is a state $j$ such that $\mathbb{P}^{\pi}(\textbf{S}_T = s | \textbf{S}_1 = j) > 0$. This implies $x_{\pi, T}(s) = \sum_{i \in \mathbb{S}}\alpha(i)\mathbb{P}^{\pi}(\textbf{S}_T = s \,|\, \textbf{S}_1 = i) > 0$, and that may also be read backwards.
\end{proof}

\begin{lemma}
\label{lem2}
Let $\pi = (\pi_1, ..., \pi_{T-1}) \in \Pi$, $t = 1, ..., T-1$, $s \in \mathbb{S}$ and $a \in \mathbb{A}_s$.
Suppose $s$ is reachable at epoch $t$ under $\pi$. Then $\sum_{a' \in \mathbb{A}_s}x_{\pi, t}(s, a') > 0$ 
and
\begin{equation}
\label{q_formula}
q(a \,|\, s,\ \pi_t) = \frac{x_{\pi, t}(s, a)}{{\sum_{a' \in \mathbb{A}_s}x_{\pi, t}(s, a')}}.
\end{equation}
\end{lemma}

\begin{proof}
Positivity $\sum_{a' \in \mathbb{A}_s}x_{\pi, t}(s, a') > 0$ follows directly
from Lemma \ref{lem1}(i), so that the
right hand side of (\ref{q_formula}) is well-defined for reachable states $s$. 
We have, for each $t = 1, ..., T-1$, $s \in \mathbb{S}$ and $a \in \mathbb{A}_s$,
\begin{flalign*}
x_{\pi, t}(s, a) &= \sum_{j \in \mathbb{S}}\alpha(j)\,\mathbb{P}^{\pi}(\textbf{S}_t = s, \textbf{A}_t = a \,|\, \textbf{S}_1 = j)\\
&= \sum_{j \in \mathbb{S}}\alpha(j)\,q(a \,|\, s,\ \pi_t)\mathbb{P}^{\pi}(\textbf{S}_t = s \,|\, \textbf{S}_1 = j)\\
&= q(a \,|\, s,\ \pi_t)\sum_{a' \in \mathbb{A}_s}\sum_{j \in \mathbb{S}}\alpha(j)\,\mathbb{P}^{\pi}(\textbf{S}_t = s, \textbf{A}_t = a' \,|\, \textbf{S}_1 = j)\\
&= q(a \,|\, s,\ \pi_t)\sum_{a' \in \mathbb{A}_s} x_{\pi, t}(s, a'),
\end{flalign*}
hence the asserted equality.
\end{proof}

This has the following consequence:

\begin{lemma}
Two regular policies $\pi,\pi'$ coincide iff their state-action frequencies agree.
\end{lemma}

\begin{proof}
    Suppose $x_{\pi,t}(s,a) = x_{\pi',t}(s,a)$ for all $t,s,a$. We have to show that 
    $q(a|s,\pi_t) = q(a|s,\pi'_t)$ for all $t,s,a$. Now clearly $\pi,\pi'$ have the same reachable states
    at epoch $t$,
    because by Lemma \ref{lem1} those are the ones where the denominator in (\ref{q_formula}) does not vanish. In that case (\ref{q_formula}) is applicable to both $\pi$ and $\pi'$ and 
    therefore gives equality. Next consider a state $s$ not reachable at epoch $t$ for
    $\pi,\pi'$. Then, as both are regular policies, we have
    $q(1|s,\pi_t)=1$ and also $q(1|s,\pi_t')=1$, hence again equality, because
    $q(a|s,\pi_t) = 0 =  q(a|s,\pi_t')$ for all $a\not=1$.
\end{proof}

This fact could also be re-stated as follows: two policies are equivalent, $\pi \sim \pi'$,
iff their state-action frequency vectors agree, respectively, iff the right hand side of (\ref{q_formula}) is the same for both, cases 0/0 included.


In a second more involved step, we will now have to show that
$\pi \mapsto x_\pi$ is surjective. This means that
every state-action frequency vector  $x \in P$ is  generated 
as $x=x_\pi$ by a regular policy $\pi\in \Pi_r$. 
This is obtained by the following proposition.

\begin{proposition}
\label{prop2}
Given $x \in  P$,  there exists a unique regular policy $\pi^x \in \Pi_r$ such that $x = x_{\pi^x}$. In particular, $\pi \mapsto x_\pi$ is surjective 
onto $P$.
\end{proposition}

\begin{proof}
Uniqueness is clear from Lemma \ref{lem2}. Existence is
proved by induction.

We say that a state $s$ is \textit{attainable} by the state-action frequency vector $x$
at epoch $t$ if $\sum_{a\in \mathbb{A}_s} x_t(s,a) > 0$, otherwise $s$ is unattainable by $x$ at epoch $t$. We prove
by induction on $t$ that for every $t$ there exists a regular policy $\pi^t$ such that $x_{t'}(s,a) = x_{\pi^t,t'}(s,a)$ for all $t'=1,\dots,  t$, all $s\in \mathbb{S},a\in \mathbb{A}_s$, such that in addition (\ref{q_formula}) holds for all  $s$ attainable at $t'=1,\dots,t$.

To start the induction, note that by condition (a) we have
$\sum_{a\in \mathbb{A}_j} x_1(j,a) = \alpha(j) > 0$  for every $j$, hence every $j$ is attainable
at $t=1$ and we can put
$$
q(a|s,\pi_1^1) = \frac{x_1(s,a)}{\sum_{a'\in \mathbb{A}_s} x_1(s,a')},
$$
which fixes $\pi_1^1$.
Defining $\pi_2^1,\dots,\pi_{T-1}^1$ arbitrarily gives a policy $\pi^1$ which may without loss 
be assumed regular, because if it is not regular at some $(s,t)$ with $t >1$, we can apply the regularization procedure described earlier.
We check that
$x_{\pi^1,1}(s,a) = x_1(s,a)$ for all $s$ and $a$. Now
\begin{flalign*}x_{\pi^1, 1}(s, a) &= \sum_{j \in \mathbb{S}}\alpha(j)\,\mathbb{P}^{\pi^1}(\textbf{S}_1 = s, \textbf{A}_1 = a \,|\, \textbf{S}_1 = j) \\
&= \sum_{j \in \mathbb{S}}\alpha(j)\,q(a \,|\, s,\ \pi_1^1)\mathbb{P}^{\pi^1}(\textbf{S}_1 = s \,|\, \textbf{S}_1 = j)
\\
&= \frac{x_1(s, a)}{\sum_{a' \in \mathbb{A}_s}x_1(s, a')}\sum_{j \in \mathbb{S}}\alpha(j)\,\mathbb{P}^{\pi^1}(\textbf{S}_1 = s \,|\, \textbf{S}_1 = j) \\
&= \frac{x_1(s, a)}{\sum_{a' \in \mathbb{A}_s}x_1(s, a')}\sum_{a' \in \mathbb{A}_s}x_{\pi^1, 1}(s, a') \\
&= x_1(s, a),
\end{flalign*}
for all $s$, where the final equality is due to the fact that $\sum_{a \in \mathbb{A}_s}x_{1}(s, a) = \sum_{a \in \mathbb{A}_s}x_{\pi^1, 1}(s, a)$. Our claim therefore holds for $t = 1$.

For the induction step, let $t=2,\dots,T-1$, and suppose there exists a regular policy $\pi^{t-1}$ 
which realizes $x$ up to epoch $t-1$, that is, 
$x_{t'}(s,a) = x_{\pi^{t-1},t'}(s,a)$ for all epochs $t'=1,\dots, t-1$ and all $s,a$, such that formula (\ref{q_formula}) is satisfied
for states $s$ reachable by $\pi^{t-1}$ at some epoch $t'=1,\dots, t-1$, while
$q(1|s,\pi^{t-1}_{t'}) = 1$ for unreachable $s$ at $t'$. 


Now we have to define a
new regular policy $\pi^t$ which realizes $x$ up to epoch $t$. We let $\pi^t_{t'} = \pi^{t-1}_{t'}$ 
for epochs $t'=1,\dots, t-1$, and define
$$
q(a|s,\pi^t_t) = \frac{x_t(s,a)}{\sum_{a'\in \mathbb{A}_s} x_t(s,a')}
$$
for those states $s$ attainable by $x$ at epoch $t$. For
states $s$ unattainable at $t$ we define
$q(1|s,\pi_t^t)=1$. We prolong $\pi_{t+1}^t,\dots,\pi_{T-1}^t$ arbitrarily so that $\pi^t$
is regular. We now have to show that $x_{t'} = x_{\pi^t,t'}$ for $t'=1,\dots, t$, and that (\ref{q_formula}) is satisfied up to epoch $t$
at reachable states $s$. Since $\pi^t$ agrees with $\pi^{t-1}$ up to $t-1$, it remains to
check $x_t = x_{\pi^t,t}$. 
We have the sum equality
\begin{align*}
    \sum_{a'\in \mathbb{A}_s} x_t(s,a') &= \sum_{j\in \mathbb{S}} \sum_{a'\in \mathbb{A}_j} p_{t-1} (s|j,a') x_{t-1}(j,a') 
    \qquad\qquad \mbox{ (by (b))}\\
    &= \sum_{j\in \mathbb{S}} \sum_{a'\in \mathbb{A}_j} p_{t-1} (s|j,a') x_{\pi^{t-1},t-1}(j,a')  \qquad \mbox{ (by induction hyp.)}\\
    &=\sum_{j\in \mathbb{S}} \sum_{a'\in \mathbb{A}_j} p_{t-1} (s|j,a') x_{\pi^t,t-1}(j,a') \qquad \;\,\,\,\mbox{ (by definition of $\pi^t$})\\
    &= \sum_{a'\in \mathbb{A}_s} x_{\pi^t,t}(s,a').
    \end{align*}

    Now if $s$ is not attainable by $x$ at $t$, the left hand side equals zero, which implies that the
    right hand side also equals zero, so that the
    $x_{\pi^t,t}(s,a')$ are all zero. Hence the claimed equality is satisfied in that case.

    Consider the case where $s$ is attainable by $x$ at epoch $t$. Then 
    \begin{align*}
    x_{\pi^t,t}(s,a) &= \sum_{j \in \mathbb{S}} \alpha(j) \mathbb P^{\pi^t}(\textbf{S}_t=s,\textbf{A}_t=a|\textbf{S}_1=j)\\
    &= \sum_{j\in \mathbb{S}} \alpha(j) q(a|s,\pi^t_t) \mathbb P^{\pi^t}(\textbf{S}_t=s|\textbf{S}_1=j)\\
    &= \frac{x_t(s,a)}{\sum_{a'\in \mathbb{A}_s} x_t(s,a')} \sum_{j\in \mathbb{S}} \alpha(j) \mathbb P^{\pi^t}(\textbf{S}_t=s|\textbf{S}_1=j)\\
    &=  \frac{x_t(s,a)}{\sum_{a'\in \mathbb{A}_s} x_t(s,a')} \sum_{a'\in \mathbb{A}_s} x_{\pi^t,t}(s,a')\\
    &= x_t(s,a)
    \end{align*}
    where the last line uses the sum equality above. Therefore $\pi^t$ is as claimed, because
    (\ref{q_formula}) is satisfied by construction. 

    Having completed the induction,
    $\pi^{T-1}$ is a regular policy which realizes $x$ up to epoch $t=T-1$.  But since constraint (c) is explicit, $\pi^{T-1}$
    gives the desired equality also at the terminal epoch $T$, so $\pi^{T-1}$ realizes $x$ up to $T$, and
    is thus the desired
    $\pi^x$ with $x = x_{\pi^x}$.
\end{proof}

This
establishes the desired bijection $\pi \mapsto x_\pi$ between
$\Pi_r$ and $P$.  Moreover, the return formula
\begin{equation}
\label{pi_x_equation}
q(a \,|\, s,\ \pi^{x}_t) = \frac{x_t(s, a)}{\sum_{a' \in \mathbb{A}_s}x_t(s, a')}
\end{equation}
gives the explicit inverse $x\mapsto \pi^x$ to
$\pi \mapsto x_\pi$. The latter is summarized by the shorthand
$$
x_{\pi^x} = x \;\mbox{ and }\; \pi^{x_\pi} = \pi.
$$

This inversion formula may also be applied to recover non-regular
policies from state-action frequency vectors in the following way. 
When $s$ is not attainable by $x$
at epoch $t$, or equivalently, not reachable by $\pi^x$ at $t$, then the right hand side of the formula reads $0/0$ for all actions $a \in \mathbb{A}_s$. We interpret this to mean that the probability distribution $\{q(a | s, \pi^{x}_t)\}_{a \in \mathbb{A}_s}$ can be defined arbitrarily.
Regular policies $\pi^x$ are those where this distribution is concentrated on $a=1$. 
All other choices of distribution produce policies equivalent to $\pi^x$.
Yet another way to express this is to say that, given $x$, the return formula fixes
$\pi^x$ up to equivalence $\sim$.

The bijectivity of $\pi \mapsto x_\pi$ from $\Pi_r$ onto $P$ having been established,
we may now propose the sought vector
linear program
\begin{equation}
\tag{\mbox{${\rm vLP}$}}
\begin{aligned}
\label{LP}
\operatorname{V-max} \quad & \sum_{t = 1}^{T-1}\sum_{s \in \mathbb{S}}\sum_{a \in \mathbb{A}_s} x_{t}(s, a)R_t(s, a) + \sum_{s \in \mathbb{S}}x_{T}(s)R_T(s), & \textrm{s.t.} \quad & x \in P,\\
\end{aligned}
\end{equation} 
which is equivalent to the vector maximization program (\ref{program}), because policies $\pi\in \Pi_r$ are in one-to-one correspondence
with feasible solutions
$x=x_\pi \in P$, and the objective
$\sum_{s\in \mathbb{S}} \alpha(s) v^\pi(s)$
of (\ref{program}) is carried into the objective
of (\ref{LP}) under (\ref{eq1})-(\ref{eq3}).
These facts are summarized in the following main theorem.
\begin{theorem}
\label{centralthm}
Let $V'$ denote the value set of program {\rm (\ref{LP})}, i.e., \begin{equation*}V' = \biggl\{\sum_{t = 1}^{T-1}\sum_{s \in \mathbb{S}}\sum_{a \in \mathbb{A}_s} x_{t}(s, a)R_t(s, a) + \sum_{s \in \mathbb{S}}x_{T}(s)R_T(s):\ x \in P\biggr\}.\end{equation*} Then $V' = V$, where $V$ is the value set of {\rm (vMDP)}. Furthermore, a policy $\pi\in \Pi$ is efficient for {\rm (vMDP)} if and only if $x_{\pi}\in P$ is an efficient solution of {\rm (\ref{vLP})}. 
\end{theorem}

\begin{proof}
    Observe that $P$ is nonempty as a result of Proposition \ref{prop1} and of $\Pi$ being nonempty. Let $v' = \sum_{t = 1}^{T-1}\sum_{s \in \mathbb{S}}\sum_{a \in \mathbb{A}_s} x_{t}(s, a)R_t(s, a) + \sum_{s \in \mathbb{S}}x_{T}(s)R_T(s)$ for some $x \in P$. By Proposition \ref{prop2}, there exists a policy $\pi$ such that $x = x_{\pi^{x}}$. Therefore, $v' = \sum_{s \in \mathbb{S}}\alpha(s)v^{\pi^{x}}(s)$ by Equation (\ref{eq1}), and so $v' \in V$. As a result, $V' \subseteq V$. Seeing that the converse inclusion is a simple consequence of Equation (\ref{eq1}) and Proposition (\ref{prop1}), the theorem is proved.
\end{proof}

In the special case where the model is a real-valued MDP ($k = 1$), Theorem \ref{centralthm} effectively states that a policy $\pi$ is optimal --- which is to say, $\sum_{s \in \mathbb{S}}\alpha(s)v(\pi) = \max_{\pi' \in \Pi}\sum_{s \in \mathbb{S}}\alpha(s)v(\pi')$ --- if and only if $x_{\pi}$ is an optimal solution of the ordinary linear program \begin{equation*}\max \biggl\{\sum_{t = 1}^{T-1}\sum_{s \in \mathbb{S}}\sum_{a \in \mathbb{A}_s} x_{t}(s, a)R_t(s, a) + \sum_{s \in \mathbb{S}}x_{T}(s)R_T(s): x \in P \biggr\}.\end{equation*} 

We end this section by observing that, naturally, by reshaping $x$ into an $n$-dimensional column vector,
program
(\ref{LP}) takes 
the canonical form of Section \ref{sect_problem}, namely
\begin{equation*}
\begin{aligned}
\operatorname{V-max} \quad & Cx, & \textrm{s.t.} \quad & x \in P,\\
\end{aligned}
\end{equation*} 
where 
$C \in \mathbb{R}^{k \times n}$ serves to rewrite
the objective (\ref{eq1}) as $Cx$, while $P=\{x\in \mathbb R^n: Ax=b, x\geqq 0\}$ represents
the equality constraints (\ref{eq2}), (\ref{eq3}).  
This involves  $m = S\cdot T$ equality constraints
for $n = (T-1)K + S$ decision variables, where $K=\sum_{s\in \mathbb{S}} k_s$. 
The constraint set
$P$ 
is thereby revealed
as a polyhedron,   whose structure is much more convenient than the
complicated structure of the policy set $\Pi$.
While the explicit definition of $A$, $b$ and $C$ is given in the Appendix, the dimensions and rank of $A$ will be of importance, so we state the following

\begin{lemma}
\label{A_rank}
The constraint matrix $A$ obtained through the vectorization of equations {\rm (a)-(c)} in Proposition
{\rm \ref{prop1}} is of size $m \times n$
with $m < n$ and has full rank, i.e., $\rank(A) = m$.
\end{lemma}

\begin{proof}
If we had $m \geqq n$, it would follow that $S \geqq \sum_{s \in \mathbb{S}}k_s =: K$, a conclusion incompatible with the fact that $k_s > 1$ for at least one $s \in \mathbb{S}$. Thus, $m < n$.

Our claim concerning the rank will follow when we show that the operator defined through (\ref{eq2}), (\ref{eq3})
is surjective. That is to say, the system
\begin{itemize}
\item[(a)] $\displaystyle\sum_{a\in \mathbb{A}_j} x_1(j,a) = \alpha(j)$ for all $j\in \mathbb{S}$,
\item[(b)] $\displaystyle\sum_{a\in \mathbb{A}_j} x_{t+1}(j,a) = \sum_{s\in \mathbb{S}}\sum_{a\in \mathbb{A}_s}
p_t(j|s,a) x_t(s,a) + \beta(t,j)$ for all $j\in \mathbb{S}$ and $t=1,\dots,T-2$,
\item[(c)] $x_T(j) = \displaystyle\sum_{s\in \mathbb{S}} \sum_{a\in \mathbb{A}_s} p_{T-1}(j|s,a) x_{T-1}(s,a) + \gamma(j)$ for all $j\in \mathbb{S}$
\end{itemize}
must have a solution for arbitrary $(\alpha,\beta,\gamma)\in
\mathbb R^S\times \mathbb R^{(T-2)\times S} \times \mathbb R^S$.
Now clearly (a) is under-determined, as there are $S$ constraints 
and $k_j$ unknowns for every $j\in \mathbb{S}$, giving a total of $K>S$ unknowns. 
So, we can fix $x_1$
to satisfy (a). Then, using (b) first from $t=1$ to $t=2$,
the right hand side being now fixed to a value $\widetilde{\beta}(1,j)$ 
for every $j\in \mathbb{S}$ due to
$x_1$ being fixed, we have yet another under-determined  
system of the same structure
$$
\sum_{a\in \mathbb{A}_j} x_2(j,a) = \widetilde{\beta}(1,j), \quad j\in \mathbb{S},
$$
which we can again solve for $x_2$. Proceeding in this way from $x_t$ to $x_{t+1}$ using (b),
we arrive at fixing $x_{T-1}$. In the final step, since
(c) is explicit, we just have to read off the value $x_T$ from the values of the right hand side of (c), now fixed through the previous steps.
\end{proof}

\noindent The proof does not require the explicit form of $A$, which under the
vectorization of $x$ given in the Appendix shows an interesting sparsity pattern. 
\\

\section{Consequences of the main theorem}
\label{sect_consequences}


The analysis leading to Theorem \ref{centralthm} has four major implications. 
First, the efficient value set of program (\ref{program}) is identical to that of program (\ref{LP}). Second, there is a one-to-one correspondence between regular efficient policies $\pi \in \Pi^E_r$
and efficient solutions $x\in P_E$ of (\ref{LP}). 
Third, we can use Equation (\ref{pi_x_equation}) to generate an efficient policy from any efficient solution of (\ref{LP}). 
Fourth, the equivalence between the two programs allows the use of vector linear optimization techniques to solve (\ref{program}) via (\ref{LP}). We give numerical illustrations of one of these techniques in Section \ref{sect_numeric}.



As a first important consequence of the main theorem, we now
establish that (\ref{program}) admits efficient policies. 
Using the one-to-one correspondence between $P$ and $\Pi_r$, respectively, the
equivalence between (\ref{program}) and (\ref{vLP}), this boils down to finding an efficient solution 
for (\ref{LP}). This can now be derived from the following
general

\begin{lemma}
\label{existence}
    Suppose the value set of a vector linear program \begin{equation*}
\begin{aligned}
\operatorname{V-max} \quad & Cx, & \textrm{s.t.} \quad Ax = b,\ x \geqq 0\\
\end{aligned}
\end{equation*}is nonempty and bounded above. Then the program admits efficient solutions.
\end{lemma}

\begin{proof}
    Let $C(P) = \{Cx: Ax=b, x\geqq 0\}$ be the program's value set. Then, $C(P)$ is bounded above
    by hypothesis, and, being 
    a polyhedron, is closed. Therefore, $C(P)$ admits a maximal element
    $\bar{v}$, for which we find $\bar{x}\geqq 0$ with $A\bar{x}=b$ such that $C\bar{x}=\bar{v}$.
\end{proof}


This ensures the existence of efficient policies $\pi^*\in \Pi^E$, because 
there exist efficient solutions $x^*$ of (\ref{vLP}).
Namely,
the polyhedron
$P$, while in the first place nonempty due to  $\Pi \not=\emptyset$, 
is bounded, as the $x_{\pi,t}(j,a)$ and $x_{\pi,T}(j)$ represent probabilities. Then its value set $\{Cx: x\in P\}$ is also bounded, so that the assumptions of Lemma \ref{existence} are met. 

We may obtain slightly more information.

In Section \ref{sect_problem}, we indicated a relationship, pointed out by Evans and Steuer \cite{7}, between efficient solutions of a vLP and optimal solutions
of scalar LPs where the objective is a positive linear combination of the original criteria $c_i^Tx$. In application of this relationship, and as a consequence of Theorem \ref{centralthm}, we derive another necessary and sufficient condition for the efficiency of a policy.

\begin{corollary}
\label{cor3}
A policy $\pi$ is efficient if and only if there are $k$ 
scalars $p_1>0, ..., p_k>0$ such that $x_{\pi}$ is an optimal solution of the linear program $\max\{(p^{T}C)x: x \in P\}$.
\end{corollary}

\begin{proof}
We know from Theorem \ref{centralthm} that $\pi$ is an efficient policy if and only if $x_{\pi}$ is an efficient solution of (\ref{LP}). Evans and Steuer's Corollary 1.4 offers the following characterization of efficiency: a point $x^{0} \in P$ is efficient for (\ref{LP}) if and only if there exist $p_1, ..., p_k > 0$ such that $(p^{T}C)x^{0} = \max\{(p^{T}C)x: x \in P\}$. This proves the desired equivalence. 
\end{proof}

Recall the following common linear programming terminology. Let the variable index set $\{1, ..., n\}$ be 
partitioned into a subset $B$ of size $m$ and its complement $N$. 
If the square sub-matrix $A_B$ of $A$ with columns in $B$ has maximal rank $m$, we say it is a \textit{basis} of the vLP. The variables in $B$ are then called the \textit{basic variables}. The system $Ax = b$ can be written as $Ax = A_Bx_B + A_Nx_N = b$, with $A_N$ of size $m \times (n-m)$ and $x=(x_B,x_N)$ partitioned accordingly. The \textit{basic solution} associated with $B$ is the unique solution $x$ of $Ax = b$ where the non-basic variables are set to zero, namely $x = (A_B^{-1}b, 0)$. We say $x$ is basic \textit{feasible} if $x_B = A_B^{-1}b \geqq 0$. A basic feasible solution is \textit{degenerate} if one or more basic variables equal zero. The set of basic feasible $x$ will be denoted $P_B$,
and these determine the vertices of the polyhedron $P$.} 

It is well known that if an ordinary linear program has an optimal solution, then it has at least one optimal basic feasible solution. Corollary \ref{cor3} has therefore the following consequence: 

\begin{proposition}
\label{prop5}
Program {\rm (\ref{LP})} has efficient solutions, and equivalently, {\rm (\ref{program})} has  efficient policies. Moreover, there exist efficient
basic feasible solutions of {\rm (\ref{vLP})}.
\end{proposition}

\begin{proof}
Let $x^0$ be efficient for (\ref{LP}), or equivalently, $\pi^0=\pi_{x^0}$ an efficient policy. Then, by Corollary \ref{cor3}, $x^0$ is an optimal solution
of a scalar linear program   $\max\{(p^{T}C)x: x \in P\}$ with suitable weights $p_i > 0$. 
As the latter is in canonical form, it
admits also an optimal basic feasible solution, say $x^*$. But by Corollary \ref{cor3} backwards, $x^*$ is 
efficient for (\ref{LP}), hence
the claim.
\end{proof}

This raises obviously the question  which  policies correspond to basic feasible solutions, that is, to vertices of $P$. We expect these to be the deterministic ones,  
and this will be clarified in the next two sections.
\\

\section{Deterministic policies -  regular case}
\label{sect_regular_case}

We relate the basic feasible solutions of (\ref{vLP}) 
to deterministic policies, starting with the regular case. 
Observe that for the scalar infinite-horizon case the corresponding question is discussed in
\cite[p. 245]{16}, where  
extreme points are shown to correspond to {\it stationary} deterministic policies. 
This is no longer true when optimizing over a finite horizon. Indeed, as we shall see, there is a one-to-one correspondence between the extreme points of $P$, respectively,
basic feasible solutions of the vLP, and {\it all} deterministic policies, stationary or otherwise.


\begin{lemma}
\label{bfs_unique_action}
Suppose the vMDP is regular. Let $x$ be a basic feasible solution of {\rm (\ref{LP})}. Then for each $t = 1, ..., T-1$ and $j \in \mathbb{S}$, there exists an action $a' \in \mathbb{A}_j$ such that $x_t(j, a') > 0$ and $x_{t}(j, a) = 0$ for all $a \neq a'$ in $\mathbb{A}_j$.
\end{lemma}

\begin{proof}
By Lemma \ref{A_rank}, $x$ contains at most $m$ non-zero components. Its $N$ terminal components, $x_T(1)$ through $x_T(N)$, are all non-zero by Lemma \ref{lem1}(ii). Thus, among the remaining $n - N$ components, at most $m - S= (T-1)\cdot N$ are non-zero. Now Lemma \ref{lem1}(i) implies that, for each $t$ and $j$, there is at least one action $a' \in \mathbb{A}_j$ satisfying $x_{t}(j, a') > 0$. 
But there can be only one such action for each $t$ and $j$; otherwise, we would have more than $(T-1)\cdot N$ non-zero variables within the non-terminal portion of $x$. Thus, $a'$ is unique for each $t$ and $j$, and that means $x_t(j,a') > 0$ and $x_t(j, a) = 0$ for all $a \neq a'$ in $\mathbb{A}_j$.
\end{proof}


\begin{corollary}
\label{cor4}
If the vMDP is regular, then all basic feasible solutions of {\rm (\ref{LP})} are non-degenerate. In other words, they contain exactly $m$ non-zero components.
\end{corollary}

\begin{proof}
The assertion follows directly from Lemmas \ref{lem1}(ii) and \ref{bfs_unique_action}.
\end{proof}

As is well known, basic feasible solutions represent extreme points,
or vertices, of the constraint polyhedron $P$. 
When degeneracy occurs in the matrix representation of $P$, a vertex
may allow several basic feasible solutions, which as is well known
necessitates the use of anti-cycling rules in the simplex algorithm. Remarkably therefore, Corollary 
\ref{cor4} tells us that LP-degeneracy is impossible in (\ref{vLP}) if the vMDP is regular. In that case basic feasible solutions are in one-to-one correspondence with the vertices of $P$.

\begin{corollary}
\label{cor5}
Suppose $x$ is a basic feasible solution of {\rm (\ref{LP})}. If the vMDP is regular, then the policy $\pi^x$ is deterministic.  
\end{corollary}

\begin{proof}
Lemma \ref{bfs_unique_action} together with the fact that $x = x_{\pi^x}$ implies that for each $t = 1, ..., N-1$ and $j \in \mathbb{S}$, we have $x_{\pi^x, t}(j, a') = x_t(j, a') > 0$ 
for a single action $a' \in \mathbb{A}_j$ and $x_{\pi^x, t}(j, a) = 0$ for all $a \neq a'$. By Lemma \ref{lem2}, 
$\pi^x$ must be deterministic. 
\end{proof}


By Proposition \ref{prop5},  the efficient value set of (\ref{LP}) is non-empty, and
there is at least one efficient vertex. 
Corollary \ref{cor5} says that if $x$ is basic feasible, hence a vertex, $\pi^x$ is deterministic. Finally, we know from the main theorem that 
$x \in P$ is efficient for (\ref{LP}) if and only if $\pi^x$ is an efficient policy. 
From this we infer the following fact about program (\ref{program}).

\begin{corollary}
\label{cor6}
Suppose the process is regular. Then program {\rm (vMDP)} 
admits a deterministic efficient policy. 
\end{corollary}

The converse of Corollary \ref{cor5} is also true:

\begin{proposition} 
\label{prop4}
Suppose the process is regular.
If $\pi$ is a deterministic policy, then $x_{\pi}$ is a basic feasible solution of {\rm (\ref{LP})}.
\end{proposition}

\begin{proof}
It follows from the main theorem that
$x_\pi$ is a feasible solution of (\ref{LP}). 
Suppose now, for the sake of contradiction, that $x_{\pi}$ is not basic. Then $x_{\pi}$ is a convex combination of 
finitely many basic feasible solutions $x^1,\dots,x^r$,
$$
x_\pi = \sum_{i=1}^r \lambda_i x^i.
$$
By Corollary \ref{cor3} there exist distinct deterministic policies
$\pi_i$ with $x^i = x_{\pi_i}$. For each $t = 1, ..., T-1$ and $s \in \mathbb{S}$, there is, by Lemma \ref{bfs_unique_action}, a unique action $a' \in \mathbb{A}_s$, depending on $\pi$, $t$ and $s$,  such that $x_{\pi, t}(s, a') > 0$. In view of the non-negativity of 
the $x^i$  it follows from $x_{\pi} = \sum_{i=1}^r\lambda_i x_{\pi_i}$ and $\lambda_i \geqq 0$ that 
\begin{equation*}x_{\pi_i, t}(s, a) =  x_{\pi, t}(s, a) = 0\end{equation*} 
for all $a \neq a'$ and all $i$. Given the non-degeneracy of the $x_{\pi_i}$ (Corollary \ref{cor4}), these $n-m$ zero components must correspond to the non-basic variables in both $x$ and the $x_{\pi_i}$. Consequently all $x_{\pi_i}$ have the same non-basic variables. Since the values of the basic variables are uniquely determined by the choice of the non-basic variables, all $x_{\pi_i}$ coincide: a contradiction. Consequently $x_{\pi}$ is a basic feasible solution of (\ref{LP}).  
\end{proof}

A characterization of deterministic policies in terms of (\ref{LP}) follows naturally from Corollary \ref{cor5} and Proposition \ref{prop4}.

\begin{corollary}
\label{cor7}
%
Suppose the process is regular. Then under
the one-to-one correspondence $x \mapsto \pi^x$ from $P$ to $\Pi$, 
vertices $x \in P_B$, or equivalently, basic feasible solutions of {\rm (${\rm vLP}$)}, are mapped to deterministic
policies $\pi^x\in  \Pi^{D}$, and conversely, every $\pi \in \Pi^{D}$ leads back to a 
basic feasible solution $x_\pi\in P_B$ under $\pi \mapsto x_\pi$. 
\end{corollary}

\begin{proof}
%
Let $\pi$ be a deterministic policy. By Proposition \ref{prop4}, 
$x_\pi$ is basic feasible for (\ref{LP}). 
Conversely, suppose $x\in P$ is basic feasible for (\ref{LP}). Writing it as $x=x_\pi$, and bearing in mind that 
$\pi = \pi^{x_\pi}$, $\pi$ is deterministic by Corollary \ref{cor5}. 
\end{proof}

It is readily seen from its construction that $\Pi^{D}$ has cardinality  $\bigl(\prod_{s = 1}^{N}k_s\bigr)^{T-1}$.

\begin{corollary}
\label{cor8}
Suppose the process is regular. Then
program {\rm (\ref{LP})} has exactly $\bigl(\prod_{s = 1}^{N}k_s\bigr)^{T-1}$ basic feasible solutions,
and equivalently, that many vertices.
\end{corollary}

The matrix-vector representation of these basic feasible solutions is given in the appendix. They correspond to specific bases $A_B$ of lower triangular shape. Although these do not exhaust the bases of (\ref{vLP}), other choices lead to {\it infeasible} basic solutions when the process is regular. 
\\

\section{Deterministic policies -  general case}
\label{sect_general_case}

We now relate deterministic policies to basic feasible solutions in the absence of regularity. 
Here the situation is more complicated, as
vertices $x\in P_B$ may have several
associated basic feasible solutions, and equivalence
of policies not distinguishable through their state-action frequencies complicates
the counting. This requires some
preparation.  We have the following

\begin{lemma}
\label{lem9}
Suppose for every $t=1,\dots,T-1$ and every $s\in \mathbb{S}$
an action $a_{st}\in \mathbb{A}_s$ is fixed. Then there exists a unique state-action frequency vector
$x\in P$ satisfying $x_t(s,a_{st}) \geqq  0$ and $x_t(s,a')=0$ for all $a'\in \mathbb{A}_s$ with $a'\not=a_{st}$.
Moreover, $x$ is a vertex of $P$. 
\end{lemma}

\begin{proof}
Let $x$ be the unique solution of the system
\begin{itemize}
    \item[(${\rm a}'$)] $x_1(j,a_{j1}) = \alpha(j)$, $j\in \mathbb{S}$,
    \item[(${\rm b}'$)] $x_{t+1}(j,a_{j,t+1}) = \sum_{s\in \mathbb{S}} p_t(j|s,a_{st}) x_t(s,a_{st})$, $j\in \mathbb{S}$,
    \item[(${\rm c}'$)] $x_T(j) = \sum_{s\in \mathbb{S}} p_{T-1}(j|s,a_{s,T-1}) x_{T-1}(j,a_{j,T-1})$, $j\in \mathbb{S}$,
    \item[(${\rm d}'$)] $x_t(s,a')=0$ for all $a'\not=a_{st}$.
\end{itemize}
Then $x\geqq 0$, and $x\in P$, because under (${\rm d}')$ conditions (${\rm a}'$) - (${\rm c}'$)
are the same as (a)-(c), which by Lemma \ref{A_rank} have a solution. The explicit form 
above shows directly that this solution is unique. The latter proves that $x$ is a vertex of $P$.
\end{proof}

When written in  matrix-vector form, vertices $x$ satisfying (${\rm a}'$) - (${\rm d}'$)
have the specific bases $A_B$ described in the Appendix. 
It can be seen either from that matrix representation or directly from the above lemma that
there are $\bigl(\prod_{s = 1}^{N}k_s\bigr)^{T-1}$ ways of selecting the $(a_{st})_{s,t}$, which 
gives an upper bound for the number of vertices generated in that manner.
As shown by
Lemma \ref{bfs_unique_action} and Corollary \ref{cor6},  a regular process  gives
precisely $\bigl(\prod_{s = 1}^{N}k_s\bigr)^{T-1}$ vertices of $P$, which are all constructed
according to (${\rm a}'$) - (${\rm d}'$), with 
$x_t(s,a_{st}) > 0$.
What complicates the
non-regular case  is that now 
other basic feasible
solutions not constructed according to (${\rm a}'$) - (${\rm d}'$) may occur, and it is a priori not clear whether these lead to new types of vertices, or just to additional bases for old ones.  
On the other hand, there may be a drop in the number of vertices accounted for by
(${\rm a}'$) - (${\rm d}'$)
due to degeneracy. These issues will be clarified below.

\begin{definition}
    Basic feasible solutions of {\rm (\ref{vLP})} formed according to {\rm (${\rm a}'$) - (${\rm d}'$)} are termed regular.
    \end{definition}

\begin{proposition}
\label{new1}
    Let $\pi\in \Pi$ be deterministic. Then $x_\pi$ is a vertex of $P$
    which has a representation as a regular basic feasible solution of {\rm (\ref{vLP})}.
\end{proposition}

\begin{proof}
    We use an approximation argument. For $\epsilon > 0$ we consider a modified vMDP
    $(\Pi, \mathbb P_\epsilon,R)$, where we change transition probabilities
by small amounts such that
$p_{t-1}^\epsilon(s|s',a') > 0$, and such that these modified probabilities
depend continuously on $\epsilon$ with
$p_{t-1}^\epsilon(s|s',a')\to p_{t-1}(s|s',a')$ as
$\epsilon \to 0$. We keep policies and rewards as is.
This makes each approximating process regular and allows us to apply the results of the previous section.

Letting $T:\Pi\to P$ denote the operator $\pi \to x_\pi$ based on (\ref{eq2}), (\ref{eq3}), we have 
similar operators $T_\epsilon:\Pi\to P_\epsilon = \{x: A^\epsilon x=b, x\geqq 0\}$
based on ${\rm (\ref{eq2})}_\epsilon$, ${\rm (\ref{eq3})}_\epsilon$ with the modified probabilities, leading to
state-action frequency vectors $x_\pi^\epsilon\in P_\epsilon$ 
parameterized by $\epsilon$. The matrices $A^\epsilon$ are obtained by the same vectorization given in the appendix, and therefore have the same 
structure. The only changes are that $P_t$-blocks are replaced  by
$P_t^\epsilon$-blocks.
By regularity, invoking Proposition \ref{prop4}, each $x_\pi^\epsilon$ is now a basic feasible solution of $P_\epsilon$, which is regular by Corollary \ref{cor8}. 
Recall that, matrix-wise, 
this means that the basis $A_B^{\epsilon}$ 
is such that one column is chosen
from every $1\times k_s$-block of each $\Sigma$-matrix. 
Note that $B$ should {\it a priori} depend on $\epsilon$, but since there are only finitely many possible bases and infinitely many $\epsilon>0$, we may without loss of generality
fix $B$, passing to a subsequence $\epsilon_i \to 0$ if necessary.

Now $A_B$, being a one-column-per-block sub-matrix in the sense of the appendix, is a feasible basis of $A$. Consequently $x = (x_B, 0) = (A_B^{-1}b, 0)$ is a regular basic feasible solution with respect to $P$. Since equations 
${\rm (\ref{eq2})}_\epsilon$, ${\rm (\ref{eq3})}_\epsilon$
defining the transformation $T_\epsilon$ depend continuously on $\epsilon$, we 
obtain (\ref{eq2}), (\ref{eq3}) in the limit and
deduce
that $x = x_\pi$. Thus, the
image $x_\pi$ of a deterministic $\pi$ is a vertex of $P$ (namely $x$), which, in addition, admits a representation as a regular basic feasible solution.
This extends Proposition \ref{prop4} to non-regular processes.  


\end{proof}

Note that fixing $B$ for a subsequence $\epsilon_i \to 0$ could mean that we had to
fix another basis $B'$ for another subsequence $\epsilon_i'\to 0$. Since the above argument works for all of these,
the limit vertex $x_\pi \in P$ could arise from different regular basic feasible solutions, so that
without regularity of the vMDP, the vLP may turn out to be degenerate. This is indeed what happens.

%
%

We also recall that there are exactly $(\Pi_{s=1}^Sk_s)^{T-1}$
deterministic policies, regardless of whether the process is regular or not. 
We know however that those equivalent in the sense of Section \ref{sect_LP} cannot be distinguished by their state-action frequency vectors, which are identical. 
Suppose a deterministic policy has $r \geqq 1$ non-reachable state-epoch pairs $(s_1,t_1),\dots,(s_r,t_r)$, where $k_{s_i} > 1$.  Then it has $(k_{s_1}-1)\cdots (k_{s_r}-1)$ equivalent 
deterministic policies.


\begin{proposition}
    Let $x$ be a vertex of $P$ with an associated regular basic feasible
    solution. 
    Then $x=x_\pi$ for a deterministic policy $\pi\in \Pi^D$, which may in addition be chosen to be regular.
\end{proposition}

\begin{proof}
    Let $x = (x_B, 0)$ be a vertex of $P$ where the columns in $B$ are chosen according to the column selection rule laid out in the appendix. 
    Then $A_Bx_B = (\alpha,0)$, $x_N=0$, and $x_B \geqq 0$. Using the same approximation as
in Proposition \ref{new1},
choose $x^\epsilon$ such that $A_B^\epsilon x_B^\epsilon = (\alpha,0)$, $x_N^\epsilon = 0$, where the same basic columns $B$ are chosen.
Then $x_B^\epsilon \geqq 0$ is redundant, and we have $x^\epsilon \to x$.

Using regularity of the approximating process,
there exists a unique $\pi_\epsilon \in \Pi$ obtained from $x^\epsilon$
by the return formula for $\epsilon$. In addition,  $\pi_\epsilon$ is a deterministic policy
by Corollary \ref{cor3}. Now since $\Pi$ can be considered a compact subset of 
$\mathbb R^T \times \mathbb R^S\times \mathbb R^K$, we can extract convergent subsequences
$\pi_{\epsilon_i} \to \pi \in \Pi$. Every such $\pi$ is deterministic, as this is preserved under
pointwise convergence. It follows also that the image of any such $\pi$ under $\pi \mapsto x_\pi$
is $x$. Here
we do not need the return formula, it suffices to use
${\rm (\ref{eq2})}_\epsilon$, ${\rm (\ref{eq3})}_\epsilon$
and pass to the limit $\epsilon \to 0$.  This means that every accumulation point of the $\pi_\epsilon$ is mapped to the same $x$. Finally, letting
$\pi_r$ be the regular policy with $\pi \sim \pi_r$ gives the same image $x_\pi = x_{\pi_r}$.
\end{proof}

\begin{corollary}
    The vMDP admits efficient deterministic policies.
\end{corollary}

Without process regularity we cannot exclude the possibility that a vertex $x\in P$
has a basic feasible representation which is {\it not} regular. 
It turns out that in this case $x$ has yet  another basic representations which {\it is} regular.

\begin{proposition}
\label{prop7}
    Every vertex $x$ of $P$ admits a representation as a regular basic feasible solution.
\end{proposition}

\begin{proof}
    We have to show that if $x$ is a vertex with a basis $B$ 
    which is not obtained as in the appendix, respectively, not according to (${\rm a}'$) - (${\rm d}'$),    then there is for this same vertex another basis representation which
    {\it does} have that form.

    We consider the approximations $(\Pi,\mathbb P_\epsilon,R)$ used before. 
    Each gives rise to a polyhedron $P_\epsilon = \{x: A^\epsilon x=b, x\geqq 0\}$, which
    approximates $P = \{x: Ax=b, x\geqq 0\}$ as $\epsilon \to 0$ due to
    $A^\epsilon \to A$ as $\epsilon \to 0$. To make this precise, 
    since the structure of the $A^\epsilon$ is such that we know that the
    $P_\epsilon$ are all non-empty, there exist for the given $x\in P$ by Hoffman's theorem \cite{11} elements 
    $x^\epsilon \in P_\epsilon$ with $x^\epsilon \to x$. While these
    $x^\epsilon$ need not be vertices, there do exist vertices $x^{\epsilon,1},\dots,x^{\epsilon,r}$
    of $P_\epsilon$
    such that $x^\epsilon = \sum_{i=1}^r \lambda_i^\epsilon x^{\epsilon,i}$ with
    $\lambda_i^\epsilon\geqq 0$ and $\sum_{i=1}^r \lambda_i^\epsilon=1$. The number $r$ may be fixed,
    as the number of vertices in every $P_\epsilon$ cannot exceed ${n \choose m}$. Since the approximating processes
    are regular, each of these $x^{\epsilon,i}$ has a basis representation involving basic
    columns $B_{\epsilon,i}$ chosen as in Lemma {\rm \ref{bfs_unique_action}}.    

    Since $\bigcup_{\epsilon > 0} P_\epsilon$ is bounded,  we may extract 
    a subsequence such that $x^{\epsilon,i} \to x^i$, $\lambda_i^\epsilon \to \lambda_i$, and in addition,
    $B_{\epsilon,i} = B_i$  for every $i$.
    Since $A^\epsilon x^{\epsilon,i} = b$, $x_{\epsilon,i} \geqq 0$, we get $Ax^i=b$, $x^i\geqq0$, i.e.,
    $x^i \in P$. Moreover $\sum_{i=1}^r \lambda_i^\epsilon x^{\epsilon,i} \to\sum_{i=1}^r \lambda_i x^i$, which gives $x=\sum_{i=1}^r \lambda_i x^i$. But $x$ is a vertex of 
    $P$, and hence cannot be a proper convex combination of elements of $P$, which implies
    $x^i=x$. Now all $x^{\epsilon,i}$ are basic feasible with the same basis ${A}_{B_i}$. Then this is also
    true for $x^i$. Hence ${A}_{B_i}$ is a basis for $x$.
    \end{proof}


\begin{corollary}
A policy
$\pi \in \Pi$ is deterministic up to equivalence ($\sim$) if and only if $x_{\pi}$ is a vertex of 
$P$.
A state-action frequency vector
$x \in P$ is a vertex if and only if the regular policy $\pi^x$ is deterministic. 
\end{corollary}

It follows that the true number of vertices of $P$ is  card$(\Pi^D{/_{\sim}})= 
{\rm card}(\Pi_r \cap \Pi^D)$, 
which in general is less than the upper bound
$(\Pi_{s=1}^Sk_s)^{T-1}$.
This rounds up our analysis of the general case. The result is that the search can be organized
in the same way regardless of whether the process is regular or not. 
The details are given in the next section.
\\

\section{An algorithm for enumerating efficient deterministic policies}
\label{sect_algo}

In \cite{6}, an algorithm is developed for enumerating the efficient basic feasible solutions of a vector linear program 
\begin{equation*}
\begin{aligned}
\operatorname{V-max} \quad & Cx, & \textrm{s.t.} \quad Ax = b,\ x \geqq 0.\\
\end{aligned}
\end{equation*} 
In this section, we describe a slightly adjusted version of this algorithm which, relying on the main theorem, we use for the search for efficient deterministic policies. 


Essential to the algorithm is the notion of \textit{adjacency}. We say that two basic feasible solutions of (\ref{LP})
are adjacent if their corresponding vertices in $P$ are joined by an
edge.  
In terms of the simplex algorithm, this means the following:

\begin{lemma} 
\label{pivot}
Let $x,x'\in P$ be vertices represented as regular basic feasible solutions with bases $B$, $B'$,
where for every $t=1,\dots,T-1$ and $s\in \mathbb{S}$, $B$ selects a unique $a_{st}\in \mathbb{A}_s$ such that
$x_t(s,a_{st}) \geq 0$, while $B'$ selects a unique $a'_{st}\in \mathbb{A}_s$ with $x_t'(s,a'_{st}) \geq 0$.
Suppose there exists a sole $(s,t)$ where $a_{st} \not = a'_{st}$, while $a_{s't'}=a'_{s't'}$ for all
other $(s',t') \not=(s,t)$. 
 Then $B \to B'$ corresponds to a simplex pivot, and either the  vertices $x,x'$ are  adjacent, or
$x=x'$. The latter may only happen when the vertex $x$ is degenerate and 
is excluded when the process is regular.
\end{lemma}


Starting from an initial efficient basic feasible solution, Ecker and Kouada \cite{6} proceed by examining the adjacent vertices of each vertex 
already identified as efficient. This requires an efficiency test to be applied at each adjacent vertex $x$ represented as the basic feasible solution $x=(x_B,0)$ with partition $(B,N)$. 
Partitioning the reward matrix $C$ in (\ref{vLP}) as $Cx = C_Bx_B + C_N x_N$, we have
 

\begin{theorem}
\label{thm_cor22}
Let $x$ be a basic feasible solution of {\rm (\ref{LP})} with basic variables $B$. Let $Q = \{i \in B: x_{i} = 0\}$. Write $R = C_{B}A_B^{-1}A_N - C_N$ and $Y = -A_B^{-1}A_N$. Then $x$ is efficient if and only if the scalar LP 
\begin{equation}
\label{test}
\max\biggl\{\sum_{j = 1}^{k}v_j: Ru + v = 0;\ Y_iu - s_i = 0,\ i \in Q;\ v \geqq 0;\ u \geqq 0;\ s \geqq 0\biggr\}
\end{equation} is bounded, where $v \in \mathbb{R}^k$, $u \in \mathbb{R}^{n-m}$, $s \in \mathbb{R}^{|Q|}$ and $Y_i$ signifies the $i$th row of $Y$ for all $i \in Q$.
\end{theorem}

This test is justified by \cite[Lemma 2.4]{7}.  
Contrary to their original statement, our case does not require the auxiliary program (\ref{test}) to be consistent, nor that its optimal value be zero. Actually, the program has always a feasible solution in $(u, v) = (0, 0)$, which can easily be shown to be optimal when $x$ is efficient. In practice this means that the simplex method, if used to solve (\ref{test}), can be started directly in Phase II.

\begin{remark}For a non-degenerate $x$, such as is invariably the case when the vMDP is regular (Corollary {\rm \ref{cor4}}), $Q$ is empty and the auxiliary LP reduces to 
\begin{equation*}
\label{test_regular}
\max\biggl\{\sum_{j = 1}^{k}v_j: Ru + v = 0;\ v \geqq 0;\ u \geqq 0 \biggr\},
\end{equation*}
with $v \in \mathbb{R}^k$ and $u \in \mathbb{R}^{n-m}$.
\end{remark}

A useful implication of Theorem \ref{thm_cor22} is that efficient deterministic policies can be characterized simply in terms of the boundedness of a single-objective LP.

\begin{corollary}
\label{cor_thm_cor22}
Let $\pi$ be a deterministic policy. Then $\pi$ is efficient if and only if the scalar LP {\rm (\ref{test})} associated with $x_{\pi}$ is bounded. 
\end{corollary}

We are now in a position to outline the algorithm. The algorithm 
maintains a list $\mathscr L$ of tasks still to be processed, and a list $\mathscr E$ where efficient 
vertices
found on the way are
kept.

\begin{algorithm}
\caption{Enumerating efficient basic feasible solutions for (vLP)\label{algo}}
\noindent
\fbox{
\begin{minipage}{.97\textwidth}
\begin{algorithmic}
\INPUT Initial efficient regular basic feasible solution for  vertex $x^{0}$.
\OUTPUT List $\mathscr E$ of all efficient vertices with associated regular basic feasible solutions.
\STEP{Initialize} 
Put
$\mathscr L=\{x^{0}\}$ and $\mathscr E = \{x^0\}$.
\STEP{Stopping} 
If $\mathscr L=\emptyset$, stop and return $\mathscr E$. Otherwise continue.
\STEP{Active vertex} 
Choose new active vertex $x\in \mathscr L$ with associated regular basic feasible solution $(B,N)$, and remove it from $\mathscr L$. 
\STEP{Test adjacent vertices} 
Given current active $x$, test all adjacent vertices $x'\not\in \mathscr E$ with associated regular basic feasible solutions $(B',N')$ for efficiency. 
Include those newly identified as efficient in $\mathscr E$ and in the list 
$\mathscr L$. Then return to step 2.
\end{algorithmic}
\end{minipage}
}
\end{algorithm}

In a preliminary phase we have to supply the algorithm with an initial efficient basic feasible solution $x^{0}$. This is achieved as follows. We choose arbitrary positive weights
$p_i$ for the rewards $c_i^Tx$ in (\ref{LP}), then we maximize their weighted linear combination. The resulting scalar program is bounded over $P$ and has an optimal basic feasible solution, which can be found using the simplex method. By Corollary 1.4 of \cite{7}, this optimal solution must be an efficient basic feasible solution of (\ref{LP}). 

To check whether, in step 4, a vertex $x'$ is  adjacent to the current active vertex $x$, we use Lemma 
\ref{pivot}.
Recall that for regular processes, vertices and regular basic feasible solutions are in one-to-one correspondence. Without regularity we are still allowed to restrict the search to regular basic feasible
solutions by Proposition \ref{prop7}, but we may encounter degenerate vertices,
which may have several regular basic feasible representations. 
Therefore,  before applying Lemma \ref{pivot}, we make sure that changing bases
$B \to B'$ gives also  a change of vertices, and to ultimately force 
this we use an anti-cycling rule.

To test adjacent  $x'$ for efficiency 
we rely on Theorem \ref{thm_cor22}. 
When the auxiliary program (\ref{test}) associated with $x'$ is bounded, we classify $x'$ as efficient and include it in the list $\mathscr L$ of tasks still to be processed. Note that the test is applied to $(B',N')$, obtained from $(B,N)$ through a pivot.


Once all efficient basic feasible solutions $x\in \mathscr E$ have been returned, the corresponding 
set of efficient deterministic policies 
$\Pi^{DE}=\{\pi^x: x \in \mathscr E\}$
is computed offline using Equation (\ref{pi_x_equation}).


\begin{proposition}
    The algorithm terminates finitely and returns the set of efficient vertices, each specified by an associated regular basic feasible solution.
\end{proposition}

\begin{proof}
    The set $\mathscr E(P)$ of efficient solutions of (vLP) is connected, as follows from
    general results \cite[Thm. 3.1, Thm. 3.2]{20} on vector optimization. 
    Whenever $x\in \mathscr E(P)$ is in the relative interior of a face $F$ of $P$, then $F \subseteq \mathscr E(P)$, and then all vertices of $F$
    are efficient vertices of $P$.    
    In consequence,
    $\mathscr E(P)$ is a finite union of faces of $P$,
    see also \cite{6, 2}.  Therefore, unless there exists a single efficient vertex $x^0$, in which case $Cx^0$ dominates the value set, any two efficient vertices can be joined by a path
    passing along edges between adjacent efficient vertices. 
    
    Now suppose there exists an efficient vertex $x \not= x^0$ which is not found by the algorithm. 
    Consider a path $x^0,x^1,\dots,x^k=x$ of adjacent efficient vertices joining $x^0$ and $x$. Since
    $x=x^k\not\in \mathscr E$, the vertex $x^{k-1}$ could never have been an element of the list $\mathscr L$, as otherwise $x^k$, being adjacent to $x^{k-1}$, would have been
    identified as efficient and added to $\mathscr L$ and $\mathscr E$, the latest when $x^{k-1}$ would have been chosen as the active vertex. Repeating the argument going backwards,
    we see that none of the intermediate elements of the path could have been added to $\mathscr L$. 
    But clearly $x^1$ {\it has} been added, because $x^0$ was the first active vertex, and $x^1$ is one of its efficient neighbours.
    \end{proof}

A crucial feature of the algorithm is that it does not necessarily test all vertices of $P$; 
rather, it confines itself to vertices adjacent to previously detected efficient ones. 


Recall from Corollary \ref{cor3} that a policy $\pi$ is efficient if and only if there exists a positive vector $p = (p_1, ..., p_k)$ such that $x_{\pi}$ solves the LP \begin{equation}
\tag{$P_p$}
\label{k_lp}
\begin{aligned}
\max \quad & p_{1}c_1^{T}x + ... + p_{k}c_k^{T}x, & \textrm{s.t.} \quad x \in P,\\
\end{aligned}
\end{equation}
where, in the notation of Section \ref{sect_problem}, $c_{i}^T$ is the $i$th row of $C$, so that $(Cx)_i=c_{i}^{T}x$ gives the value of criterion $i$ at point $x \in P$. If we view each efficient policy as realizing a particular trade-off between the criteria, then the $p_i$ indicate the relative importance of each criterion in that trade-off. It then becomes a matter of some practical interest to determine these scalars. In this regard, assume that we are given an efficient deterministic policy $\pi$, which, for instance, has been found by the algorithm of this section. Then $x_{\pi}$ is an optimal 
basic feasible solution of (\ref{k_lp}) for some positive $p = (p_1, ..., p_k)$, which we now determine. Notice that the objective function of (\ref{k_lp}) is $x \mapsto p^{T}Cx$. If $A_B$ is the basis matrix associated with $x_{\pi}$ and $A_N$ is its nonbasic complementary, then the Karush-Kuhn-Tucker conditions for LPs \cite{14} imply 
\begin{equation}
\label{KKT}
 A_N^{T}A_B^{-T}(p^{T}C)_B - (p^TC)_N \geqq 0,
\end{equation}   
where $(p^{T}C)_N$ and $(p^{T}C)_B$ are, respectively, the non-basic and the basic portions of $p^{T}C$. This is both a necessary and sufficient condition for the optimality of $x_{\pi}$. Put differently, if there is a $k$-vector $p$ such that condition (\ref{KKT}) is satisfied, then $x_{\pi}$ is optimal for (\ref{k_lp}). This motivates the following LP for the identification of the weighting factor $p$ corresponding to the (efficient deterministic) policy $\pi$:
\begin{equation}
\label{test2}
\begin{aligned}
\begin{split}
\max \quad & 0\\
\textrm{s.t.} \quad & \lambda_N  + A_N^{T}A_B^{-T}(p^{T}C)_B - (p^TC)_N= 0\\
  & \lambda_N \geqq 0  \\
  & p \geqq \epsilon,
  \end{split}
\end{aligned}
\end{equation}
with $\lambda_N \in \mathbb{R}^{n-m}$, $p \in \mathbb{R}^{k}$ and $\epsilon>0$ fixed. 
This LP will yield a positive $k$-vector $p$ such that (\ref{KKT}) is satisfied, and hence weights $p_i > 0$ which render $x_{\pi}$ optimal for (\ref{k_lp}).

Our final remark is that each of test (\ref{KKT}) and program (\ref{test2}) can be used as an efficiency test for a basic feasible $x=(x_B,0)$. An optimal solution $(\lambda_N^*,p^*)$
of (\ref{test2}) will provide weights such that $x$ is optimal for $(P_{p^*})$, with $\lambda^*=(0,\lambda_N^*)$
being an associated Lagrange multiplier for the inequality constraints $x\geqq 0$. Therefore, $p^*$ gives
the importance of each criterion $c_i^Tx$, and $\lambda_N^*$ gives the sensitivity of the active
inequality constraints $x\geqq 0$ at the solution. The Lagrange multipliers for the equality constraints
are also available as $A_B^{-T}(p^{*T}C)_B$.
\\

\section{A numerical example}
\label{sect_numeric}
We illustrate the results of the preceding sections
by way of an engineering problem discussed in \cite{22}. Suppose we have to design a two-component system where each component can be chosen from a predetermined set of alternatives. The cost of selecting a particular alternative for either component is known. We have also measurements of the reliability
of each component.
As
the cheapest designs are typically the least reliable, 
we prefer a compromise between cost and reliability, 
which requires introducing a cost for the lack of reliability, 
leading to two conflicting objectives.

We formulate this problem as a vMDP. 
The components are labeled $1$ and $2$. One component is  considered at a time, so that there are two decision epochs. The state at an epoch is the present component's label. 
Components are considered in random order, with $\alpha(s) > 0$ denoting the probability that component $s \in \{ 1, 2\}$ is designed first. We have $k_s$ design options available for component $s$, so that the action set for component $s$ is $\mathbb{A}_s = \{1, ..., k_s\}$. In our notation, we have $k = 2$, $T = 3$ and $S= 2$. Thus, in this simple example, program (\ref{LP}) has $n = 2(k_1 + k_2 + 1)$ variables and $m = 6$ constraints.

The dynamics of the problem can be modeled as follows. If component $s$ is designed first, then the probability that component $j$ will be designed second given any choice of design for $s$ equals one if $j \neq s$, and zero otherwise. That is, 
\begin{equation*}
p_1(j | s, a) = 1 - \delta_{js}
\end{equation*} 
for each $s$, $j = 1, 2$ and $a \in \mathbb{A}_s$, where $\delta_{ij}$ is the Kronecker delta. 
We let 
\begin{equation*}
p_2(j|s, a) = \frac{1}{2}
\end{equation*} 
for each $s$, $j$ and $a \in \mathbb{A}_s$, so that $\sum_{j = 1}^{2}p_2(j | s, a) = 1$ whenever $s = 1, 2$ and $a \in \mathbb{A}_s$. These terminal transition probabilities are arbitrary and have essentially no effect on the problem, since the transition to epoch $3$ occurs only after the system has been fully designed. Furthermore, as we shall see below, the terminal rewards will be defined so that the artificial state obtaining in epoch $3$ does not play a role in optimization.
  
Computing efficient policies requires a reward function reflecting the competing goals of minimizing cost and maximizing reliability. For component $s$, we denote the cost and reliability of alternative $a$ by $c_{sa}$ and $p_{sa}$, respectively. We assume functionally independent components, so that the reliability of the system as a whole is the product of the reliability of component $1$ with that of component $2$. Since rewards accrue additively, we use the logarithm of the product of the two reliabilities as a measure of system reliability. 
The vector reward for choosing design $a$ for component $s$ is then given by
\begin{equation*}
R_t(s, a) = (-c_{sa},\ \log(p_{sa})),\ t = 1, 2,
\end{equation*} with $R_3(s) = (0, 0)$, so that the value of a policy gives the vector of the expected reliability and the negative of the expected total cost of a system designed according to this policy. Indeed, with this choice of rewards, the value of a policy $\pi$ may be expressed as
\begin{equation*}
\sum_{s = 1}^{2}\alpha(s)v_2^{\pi}(s) = \sum_{s = 1}^{2}\alpha(s)\Biggl(-\mathbb{E}_{s}^{\pi}\biggl[\sum_{t = 1}^{2}c_{\textbf{S}_t\textbf{A}_t}\biggr],\ \mathbb{E}_{s}^{\pi}\biggl[\log\biggl(\prod_{t = 1}^{2}p_{\textbf{S}_t\textbf{A}_t}\biggr)\biggr]\Biggr).
\end{equation*} 

We know from the preceding sections that efficient deterministic policies exist for this problem. The reader can easily verify through Lemma \ref{unreach1}(ii) that the model is regular. In practical terms, this means that all efficient deterministic policies can be detected by our algorithm.

The algorithm was tested on randomly generated instances of the problem. Costs and reliabilities were sampled from uniform distributions with a correlation coefficient of $0.7$ for both components. The initial state was distributed according to $\alpha(1) = \alpha(2) = 0.5$. The starting basic feasible solution $x^0$ was obtained for each model following the approach outlined in the preceding section. Efficiency tests in step 4 were executed by the simplex method. 

The results are summarized in Tables \ref{table:1}, \ref{table:2} and \ref{table:3}. In Table \ref{table:1}, the instances are grouped by $k_1$ and $k_2$ into nine groups of a hundred instances each. The column $|\Pi^{ED}|$ reports the average number of efficient deterministic policies in each group. The columns $\rho_1$ and $\rho_2$ contain group averages of the empirical correlation coefficients between costs and reliabilities for component $1$ and component $2$, respectively. The standard deviations of these three sets of averages are indicated between parentheses. Table \ref{table:2} presents the data used in constructing one of the reported instances. We enumerate the efficient deterministic policies for this instance in Table \ref{table:3}. For a policy $\pi = (\pi_1, \pi_2) \in \Pi^{D}$, if $a_1 \in A_1$ and $a_2 \in A_2$ denote the unique alternatives satisfying $q(a_1 \,|\, 1, \pi_t) = q(a_2 \,|\, 2, \pi_t) = 1$ for some $t = 1, 2$, the column $\pi_t$ will read ``$(a_1, a_2)$". 
\begin{table}[!ht]
\centering
\begin{tabular}{||c c c c c c||} 
 \hline
 Group & $k_1$ & $\rho_1$ & $k_2$ & $\rho_2$ & $|\Pi^{ED}|$ \\ [0.5ex] 
 \hline\hline
 1 & 5 & .67 (\textit{.31}) & 5 & .65 (\textit{.30}) & 12.5 (\textit{3.7}) \\
 2 & 5 & .68 (\textit{.29}) & 10 & .68 (\textit{.20}) & 15.3 (\textit{4.1}) \\
 3 & 5 & .68 (\textit{.32}) & 25 & .71 (\textit{.11}) &  18.1 (\textit{4.1}) \\
 4 & 10 & .70 (\textit{.18}) & 10 & .66 (\textit{.21}) &  17.6 (\textit{3.7}) \\
 5 & 10 & .69 (\textit{.19}) & 25 & .68 (\textit{.11}) &  20.7 (\textit{4.6}) \\
 6 & 25 & .70 (\textit{.11}) & 25 & .69 (\textit{.12}) &  22.9 (\textit{6.3}) \\
 7 & 50 & .69 (\textit{.07}) & 50 & .68 (\textit{.09}) &  26.2 (\textit{6.8}) \\
 8 & 75 & .70 (\textit{.06}) & 75 & .70 (\textit{.06}) & 29.1 (\textit{8.7}) \\
 9 & 100 & .70 (\textit{.05}) & 100 & .69 (\textit{.05}) & 30.7 (\textit{7.8}) \\ [1ex]
 \hline
\end{tabular}
\caption{Results for randomly generated instances of the engineering design problem grouped by $k_1$ and $k_2$.}
\label{table:1}
\end{table}

\begin{table}[!ht]
\begin{subtable}[b]{0.45\linewidth}
\centering
\begin{tabular}{|| c c c ||}
\hline
Alternative & $c_{1a}$ & $p_{1a}$ \\ [0.5ex]
\hline\hline
$1$           & $0.70$      & $0.48$      \\
$2$           & $0.33$      & $0.21$      \\
$3$           & $0.83$      & $0.58$      \\
$4$           & $0.60$      & $0.81$      \\
$5$           & $0.29$      & $0.68$     \\ [1ex]
\hline
\end{tabular}
\caption{Component 1}
\end{subtable}
\begin{subtable}[b]{0.45\linewidth}
\centering
\begin{tabular}{|| c c c ||}
\hline
Alternative & $c_{2a}$ & $p_{2a}$ \\ [0.5ex]
\hline\hline
$1$           & $0.48$      & $0.56$      \\
$2$           & $0.42$      & $0.79$      \\
$3$           & $0.39$      & $0.46$      \\
$4$           & $0.76$      & $0.38$      \\
$5$           & $0.98$      & $0.90$     \\ [1ex]
\hline
\end{tabular}
\caption{Component 2}
\end{subtable}
\caption{Costs and reliabilities for an example with $k_1 = k_2 = 5$.}
\label{table:2}
\end{table}

\begin{table}[!ht]
\centering
\begin{tabular}{||c c c c||} 
 \hline
 Policy & $\pi_1$ & $\pi_2$ & Value\\ [0.5ex] 
 \hline\hline
 1 & (5, 2) & (5, 2) & (-0.72, -0.61) \\
 2 & (4, 2) & (5, 2) & (-0.87, -0.53) \\
 3 & (4, 2) & (4, 2) & (-1.02, -0.44) \\ 
 4 & (4, 5) & (4, 2) & (-1.30, -0.38) \\
 5 & (4, 5) & (4, 5) & (-1.58, -0.32) \\
 6 & (4, 2) & (4, 5) & (-1.30, -0.38) \\
 7 & (5, 2) & (4, 2) & (-0.87, -0.53) \\
 8 & (5, 2) & (5, 3) & (-0.70, -0.88) \\
 9 & (5, 3) & (5, 3) & (-0.68, -1.16) \\
 10 & (5, 3) & (5, 2) & (-0.70, -0.88) \\ [1ex]
 \hline
\end{tabular}
\caption{Summary of the efficient deterministic policies for the example of Table \ref{table:2}.}
\label{table:3}
\end{table}

We can see from Table \ref{table:1} that the set of efficient deterministic policies grows as the number of design options increases. Table \ref{table:3} reveals some interesting characteristics of the efficient policies in the example described in Table \ref{table:2}. First, they invariably select either the cheapest or the most reliable option for component $1$. Most policies prefer alternative $2$ for component $2$, although it is neither the cheapest nor the most reliable. This can be explained by the fact that alternative $2$ is simultaneously the second most reliable and the second cheapest option, with only a marginal differential between its cost and that of the cheapest option.
When it comes to component $1$, however, cost and reliability are more tightly correlated.

Second, the policies bear out a structural property peculiar to the 
problem. Namely, if $\pi = (\pi_1, \pi_2)$ is a non-stationary deterministic policy, then reversing the order of the decision rules yields a policy $\pi' = (\pi_2, \pi_1)$ with the same value as $\pi$. In particular, if $\pi$ is efficient, then $\pi'$ is also efficient (see policies 2 and 7, 4 and 6, 8 and 10). This property can be established formally by noting that the rewards are stationary, that the alternatives available at epoch $1$ are identical to those available at epoch $2$, and that state transitions are deterministic.
\\

\section{Conclusion}
\label{sect_conclusion}
We proposed a vector linear programming approach to the optimal control of a finite-horizon vMDP with time-heterogeneous dynamics and rewards. The optimal control problem and the vLP were shown to be equivalent. A closed form return formula allows to retrieve efficient
policies for the vMDP from efficient solutions of the vLP. 
In addition, deterministic policies were shown to correspond to vertices of the vLP constraint polyhedron. An algorithm was developed that leverages this insight for identifying all efficient deterministic policies.

While our exposition used the assumption that the initial state of the process follows a non-degenerate probability distribution $\alpha$,
simplifications are possible when, for instance, 
the process initiates in a single state $s_0 \in \mathbb{S}$. 
We would consider then the modified problem 
\begin{equation*}
\begin{aligned}
\operatorname{V-max} \quad & v^{\pi}(s_0), & \textrm{s.t.} \quad \pi \in \Pi,\\
\end{aligned}
\end{equation*}   
to which  our results extend  \textit{mutatis mutandis}. 
Based on Equation (\ref{value}), we redefine the variables in (\ref{eq2}) and (\ref{eq3}) as 
$x_{\pi, t}(j, a) = \mathbb{P}^{\pi}(\textbf{S}_t = j, \textbf{A}_t = a \,|\, \textbf{S}_1 = s_0)$ and 
$x_{\pi, T}(j) = \mathbb{P}^{\pi}(\textbf{S}_T = j \,|\, \textbf{S}_1 = s_0)$ for all $j \in \mathbb{S}$, $a \in \mathbb{A}_j$ and $t = 1, ..., T-1$. Then $v_{T}^{\pi}(s_0)$ is linear in $x_{\pi}$, and we have again an equivalent vector linear program, where the objective is obtained
in the same way as for (\ref{vLP}), with slightly different constraints. Deterministic policies that are efficient for this new problem exist and may also be 
computed by our algorithm.     


Variants of the vMDP where rewards and transition probabilities are stationary can also be treated
by our approach 
even with considerable simplifications. 
Variables of the corresponding vLP would no longer depend on time, and the constraints would reduce to those used in infinite-horizon MDPs \cite[p. 245]{16}. 
Thereby the number of variables and constraints drops to $n = \sum_{s \in \mathbb{S}}k_s + N$ and $m = 2N$, 
and considerations of process regularity are no longer needed.  
\\

\section{Appendix}

We obtain the matrix form of the linear operator (\ref{eq2}), (\ref{eq3}).
Vectorizing in the first place  $x_t(s,a)$ into vectors $x^t\in \mathbb R^K$ for every
$t=1,\dots,T-1$, where $K = k_1 + \dots + k_N$, and then $x_T(s)$ into an $N$-vector $x^T$,
we arrange
$x = (x^{1},\dots,x^{T})$ into a column  vector of length
$n=(T-1)K+N$. Then the linear operator (\ref{eq2}), (\ref{eq3}) takes the 
$m \times n$ matrix shape
\begin{align*}
A=
    \begin{bmatrix}
        \Sigma &          &  & \\
         -P_1  &  \Sigma  &        & & \\
            &      -P_2   & \Sigma      &  &  \\
            &              &\!\!\ddots & \!\!\ddots  & &\\
            &               &        &      -P_{T-2}     & \Sigma \\ 
            &               &        &             &   -P_{T-1}  & I_N
    \end{bmatrix},
\end{align*}
where $m=TN$, and 
where $\Sigma$ is a summation matrix of size $S\times K$ with $N$ blocks labelled $s\in \mathbb{S}$, each of size $1 \times k_s$. For example, for $N=3$, $k_1=2$, $k_2=3$, $k_3=2$ with $K=7$, $\Sigma$ has the form
$$
\Sigma = \begin{bmatrix}
    1&1&    \\
     & & 1&1&1&   \\
     & &  & & & 1 & 1
\end{bmatrix}.
$$
The matrices $P_{t}$ of size $S\times K$ gather the coefficients $p_t(j|s,a)$ in constraints (b) and (c)
accordingly (see Proposition \ref{prop1}). It is easily seen that each $\Sigma$ has row rank $N$, from which the rank
condition in Lemma \ref{A_rank} can also be derived. 

Those  basic sub-matrices $A_B$ we termed {\it regular} in Section \ref{sect_general_case} choose $N$ columns in each 
$\Sigma$-block,
with just
one column in each $1\times k_i$-sub-block, giving altogether $(T-1)N$ columns. For instance, in the above example we have to pick three 
such columns per $\Sigma$,  by choosing
one of columns 1,2, one of columns 3,4,5, and one of columns 6,7, with altogether $2\cdot 3 \cdot 2$
possibilities per $\Sigma$. 
These $A_B$ are lower triangular with ones in the diagonal and entries in $[-1,0]$ below the diagonal. There are $(k_1 \dots k_N)^{T-1}$ such regular
basic sub-matrices $A_B$.

Due to the triangular structure of these $A_B$ with ones in the diagonal and non-positive entries below the diagonal, and using the non-negativity of $b = (\alpha , 0)$,
we easily see that $A_B x_B = b$ implies $x_B \geqq 0$, so that regular basic matrices constructed 
are by default
basic {\it feasible}.
Simplex pivots as in Lemma \ref{pivot} correspond to exchanging just one column in one of the 
$1\times k_s$-sub-blocks of a
$\Sigma$-blocks against another column from the same $1 \times k_s$-block in the same $\Sigma$.

In view of Corollary \ref{cor8}, all basic feasible solutions in a regular process are regular, i.e.,  constructed as above. 
Note that there are basic sub-matrices
$A_B$ constructed differently, 
and those may become feasible in the absence of regularity.

While all matrices $P_t$ are typically full, the $\Sigma$ are sparse 
with only $K$ non-zeros, so that 
$A$ has $nz = (T-1)(K+NK)+N$ non-zero elements. Altogether
$A$ is therefore 
$$
\frac{(T-1)(K+NK)+N}{((T-1)K+S) TN}
$$ 
per cent filled. For $S\approx T$ this gives a sparsity of approximately $1/N$. Similarly,
basic feasible $A_B$ have $nz = (N+N^2)(T-1)+N$ non-zeros, which gives a sparsity of
$$
\frac{(N+1)(T-1)+1}{NT^2}\sim \frac{1}{T}.
$$



\noindent\rule{\linewidth}{0.4pt}
\vspace*{2cm}

\noindent ANAS MIFRANI\\
Université de Toulouse\\
Institut de Mathématiques\\
F-31062 Toulouse Cedex 9\\
France\\
Email address: anas.mifrani@math.univ-toulouse.fr\\
\\
DOMINIKUS NOLL\\
Université de Toulouse\\
Institut de Mathématiques\\
F-31062 Toulouse Cedex 9\\
France\\
Email address: dominikus.noll@math.univ-toulouse.fr\\

\end{document}